\documentclass[11pt,reqno]{amsart}
\usepackage{amsmath}
\usepackage{authblk} % affiliations

\usepackage{algorithm}
\usepackage{algpseudocode}
\usepackage[utf8]{inputenc}

\usepackage{bm}
\usepackage{graphicx}
\usepackage{subfigure} 
\usepackage{caption}
\usepackage{caption}
\usepackage{subcaption}
\usepackage[margin = 3cm]{geometry}
\usepackage{enumerate} 
\usepackage{parskip}
\usepackage{hyperref}
\usepackage{comment}
\usepackage{xcolor}
\usepackage{amsmath}
\usepackage{booktabs}
\usepackage{enumitem}

\newtheorem{assumption}{Assumption}

\usepackage[foot]{amsaddr}
\allowdisplaybreaks

\newtheorem{theorem}{Theorem}[section]
\newtheorem{lemma}[theorem]{Lemma}

\theoremstyle{definition}

\theoremstyle{remark}
\newtheorem{remark}[theorem]{Remark}
\numberwithin{equation}{section}

\numberwithin{equation}{section}

\begin{document}

\makeatletter
\renewcommand{\email}[2][]{%
  \ifx\emails\@empty\relax\else{\g@addto@macro\emails{,\space}}\fi%
  \@ifnotempty{#1}{\g@addto@macro\emails{\textrm{(#1)}\space}}%
  \g@addto@macro\emails{#2}%
}
\makeatother

\title[] % running head version
{A Generalized Energy-Based Adaptive Gradient Method for Optimization}

\author{Lin Feng $^{\dagger}$}
%\address{$^{\dagger}$Department of Mathematics, Iowa State University, Ames, IA 50011} 
\email{zero3829@iastate.edu}
\author{Hailiang Liu$^\dagger$}
\address{$^\dagger$Department of Mathematics, Iowa State University, Ames, IA 50011}
\email[Corresponding author]{hliu@iastate.edu}
\date{}

\begin{abstract}
Adaptive Gradient Descent with Energy (AEGD), introduced in \cite{liu_aegd_2025},  is a variant of Gradient Descent (GD) designed to address step size sensitivity through 
an energy-based formulation. 
AEGD is notable for its unconditional energy stability, ensuring convergence in energy regardless of the initial step size.   
In this work, we propose the Generalized Energy-Based Adaptive Gradient (gAEGD) method, which  extends AEGD by generalizing the energy function beyond the square root form to  a broader class of functions. We prove that gAEGD retains the unconditional energy stability property, remains robust to step size selection, and exhibits a two-phase adaptive dynamic: the effective step size first adjusts adaptively, then stabilizes within a range that guarantees decay of the objective function values. We establish an optimal convergence rate of $O\left(\frac{1}{k}\right)$ for finding an $\epsilon$-stationary point, along with improved convergence rates for the objective gap under a local Kurdyka–\L ojasiewicz (KL) condition. {Empirical results support the theoretical analysis and indicate that the generalized energy-based approach preforms effectively and reliably   for a broad range of optimization problems.}
\end{abstract}

\maketitle

\medskip

\section{Introduction}
This work continues the project initiated in \cite{liu_aegd_2025}, which aims to develop efficient energy-based adaptive gradient methods for optimization problems. In this paper, we focus on the unconstrained minimization problem:
\begin{align}
    \min_{x \in \mathbb{R}^d} f(x),
\label{target}
\end{align}
where function $f$ is differentiable and bounded below by $f^* := \inf\limits_{x\in \mathbb{R}^d}f(x)$. 
Despite its simple appearance,  solving this problem can be challenging --  particularly in non-convex or ill-conditioned settings -- where the selection of step sizes 
plays a crucial role. Building upon the previous work, we address  the problem of adapting step sizes within the same framework. 
A general discussion of this issue, along with background references,  is provided  in the introduction of  \cite{liu_aegd_2025}.

AEGD (Adaptive Gradient Descent with Energy) \cite{liu_aegd_2025} is an adaptive method designed specifically for the minimization problem above. It incorporates  an energy-based formulation by defining 
$$
F(x) = \sqrt{f(x) + c},
$$
where $c \in \mathbb{R}$ is chosen such that $f^* + c > 0$. The algorithm proceeds with the following updates:  
\begin{subequations}
\begin{align}
    r_{k+1} &= \frac{r_k}{1+2\eta \|\nabla F(x_k)\|^2},\\
    x_{k+1} &= x_{k} - 2\eta r_{k+1}\nabla F(x_k), 
\end{align}
\label{AEGD}
\end{subequations}
where $\nabla F_k = \frac{\nabla f(x_k)}{2\sqrt{f(x_k) + c}}$, and $r_k$ is the energy parameter,  initialized as  $r_0 = F(x_0)$.

A key feature of AEGD is its dynamic step size, governed  by the energy variable $r_k$, 
which ensures unconditional energy stability -- a guarantee of energy descent regardless of the choice of step size $\eta>0$. 
The effective step size at iteration $k$ is given by  
$$
\eta_k = \eta\frac{r_{k+1}}{\sqrt{f(x_k) + c}} = \eta\frac{r_{k+1}}{F(x_k)}.
$$
 
In this work, we extend AEGD to incorporate a broader class of energy functions beyond the square root energy function. Specifically,  
let $\hat{F}: \mathbb{R}^+\to \mathbb{R}^+$ be a smooth, strictly increasing, and concave function. We then define a generalized AEGD (called gAEGD) %energy-based adaptive gradient method 
with the following updates: 
\begin{subequations}
\begin{align} 
    r_{k+1} &= \frac{r_k}{1+\eta\frac{F'_k}{F_k}\|\nabla f(x_k)\|^2},\\
    x_{k+1} &= x_{k} - \eta \frac{r_{k+1}}{F_k}\nabla f(x_k),
\end{align}
\label{updating rule0}
\end{subequations}
where the auxiliary terms are defined as
$$
F_k := \hat{F}(f(x_k) + c), \quad F'_k := \hat{F}'(f(x_k) + c),
$$
and $c\in \mathbb{R}$ is chosen such that $f^*+c>0.$ A detailed derivation of such generalized AEGD method is presented in Section \ref{extension of AEGD}.

This scheme reduces to original AEGD  when $\hat{F}(s) = \sqrt{s}$,
and retains the advantages of energy stability and adaptive 
stepping. Notably, setting $\hat{F}(s) = log(s+1)$ yields a variant we call Adaptive Logarithmic Energy Gradient Descent (ALEGD), which shows improved empirical performance.  

We present a comprehensive theoretical analysis for the proposed generalized method, including stability, convergence, and complexity analyses. In particular, we demonstrate that the method adaptively adjusts its effective  step size to guarantee descent or maintain stability. In addition to theoretical insights, we empirically validate the effectiveness of the algorithm on a range of convex and non-convex optimization problems.

\subsection{Main Contribution.}
Our key contributions are summarized below: 
\begin{itemize}
    \item \textbf{Generalized energy-based framework.} 
    We extend AEGD by  replacing the specific  choice $\hat{F}(s) = \sqrt{s}$ with a broader class of smooth, strictly increasing, and concave functions.  This generalization yields a flexible family of algorithms with unconditional stability and other desired properties. 
    
    \item \textbf{Unconditional energy stability.} We prove that gAEGD preserves the unconditional energy stability property: the energy parameter $r_k$ decreases in $k$ and converges to a finite limit $r^*$ as $k \rightarrow \infty$, regardless of the initial step size $\eta$. 

    \item \textbf{Adaptive two-stage behavior.} We identify a distinct two-stage behavior inherent in the method. For any desired accuracy $\epsilon$, one of two outcomes occurs within a finite number of iterations: either the gradient norm $\|\nabla f(x_k)\|^2$ falls below  $\epsilon$, or the energy parameter $r_k$ drops below any given positive threshold. This adaptive mechanism ensures that the effective step size $\eta_k$ first adjusts adaptively,  and then stabilizes within a range that guarantees the decay of the function values.      
    \item \textbf{Non-asymptotic convergence guarantees.} Under mild conditions on the energy function $\hat{F}$, including smoothness, monotonicity, concavity, and asymptotic decay --  
    we establish a non-asymptotic convergence result. Specifically, we show
    $$
    \min_{k \leq N} \|\nabla f_k\|^2 < \epsilon
    $$
    within at most $ O(\frac{1}{\epsilon})$ iterations, for any step size $\eta > 0$ and initial energy $r_0 > 0$, provided $c$ is  sufficiently large.

    \item \textbf{Convergence rates under the local KL condition.} When the objective function $f$ satisfies a local Kurdyka–\L ojasiewicz (KL) condition, we show that   the iterates $\{x_k\}$  converges to either a saddle point or a local minimizer of $f$, and we derive explicit convergence rates for the objective gap based on the KL exponent. 

    \item \textbf{Empirical validation and evaluation.} 
    %\red{
    We validate the proposed framework through experiments on standard convex and non-convex benchmarks, confirming key theoretical properties such as its two-stage adaptive behavior. Across a set of applications, the results demonstrate the effectiveness and robustness of the generalized energy-based approach for a broad range of optimization problems. Among the energy functions considered, the logarithmic form  $\hat{F}(s) = \log(s+1)$ consistently outperforms AEGD in many settings.
    %}
\end{itemize}

%%%%%%%%%%%%%%%%%%%%%%%%%%%%%%%%%%%%%%%%%%%%%%%
\subsection{Related Works}  
Gradient-based methods have long served as fundamental tools for unconstrained optimization \eqref{target} due to their simplicity and ease of implementation.  Classical gradient descent (GD) ensures convergence under small step sizes, and has been widely analyzed for both convex and non-convex settings.     
For instance, for smooth, non-convex objective functions  with an L-Lipschitz continuous gradient, Nesterov \cite{nesterov_introductory_2004} showed that GD with a fixed step size of $\frac{1}{L}$ can find an $\epsilon$-stationary point at a rate of $O(\frac{1}{k})$.  Carmon et al. \cite{carmon_lower_2020} later proved that this rate is the best possible dimensional-free worst-case convergence rate achievable by first order methods. 

Recent efforts have explored adaptive gradient methods, such as AdaGrad \cite{duchi_adaptive_2011}, RMSProp \cite{hinton_neural_2012} and Adam \cite{kingma_adam_2014}, 
to address key challenges in machine learning and deep learning, such as feature sparsity and sensitivity to learning rate selection. 
These methods dynamically adjust step sizes based on local curvature or accumulated gradient information. 
While effective in practice, many such methods lack rigorous convergence guarantees in general non-convex settings and do not explicitly ensure  energy stability. To address these limitations, various extensions have been proposed -- some aiming to strengthen theoretical convergence guarantees \cite{reddi_convergence_2018, zaheer_adaptive_2018}, while others focus on  enhancing generalization performance \cite{bassily_private_2014, luo_adaptive_2018, zhuang_adabelief_2020}.  

AEGD \cite{liu_aegd_2025} offers a novel approach by adapting the step size through an energy-based framework. A key advantage of AEGD is its unconditional energy stability. Several extensions \cite{liu_adaptive_2022, liu_anderson_2024, liu_adaptive_2025} have been developed to enhance its adaptability and performance in more complex settings. 
 Liu and Tian \cite{liu_adaptive_2022} introduced stochastic gradients and momentum to adapt AEGD for stochastic objectives, while Anderson acceleration was incorporated to speed up convergence \cite{liu_anderson_2024}. More recently, preconditioning techniques have been applied to handle constrained optimization \cite{liu_adaptive_2025}. 
In contrast to these extensions, which rely on external acceleration techniques, we advance AEGD by generalizing the square root energy function to a broader class, resulting in a flexible family of algorithms capable of adapting to various objective types.  
We show that gAEGD achieve the optimal convergence rate  of $O(\frac{1}{k})$ for $\|\nabla f(x_k)\|^2$ (i.e., for finding stationary points) with fewer restrictions on the initial step size $\eta$. Notably, its  convergence remains  robust as long as the parameter $c$ is chosen sufficiently large. 

While these extensions enrich the design of energy-based adaptive methods, achieving not just convergence but a desirable  convergence rate requires additional  structural assumptions on the objective function. 
A widely used condition is the Polyak–\L{}ojasiewicz (PL) condition, originally introduced by Polyak in 1963 \cite{polyak_gradient_1963}. The PL condition relaxes the requirement of strong convexity while still ensuing the linear convergence of gradient descent \cite{karimi_linear_2016}. It has been shown that AEGD achieves the same linear convergence rate for the function value gap $\|f(x_k) - f(x^*)\|$ under the PL condition \cite{liu_aegd_2025}. In this work, we analyze the convergence of gAEGD under a more general assumption: a local Kurdyka-\L{}ojasiewicz (KL) condition given in Section \ref{convergence rates}. First  introduced by \L ojasiewicz in 1963 \cite{lojasiewicz_geometrie_1993}, the KL condition generalizes the PL condition, which it includes as a special case. Its  broader applicability  stems from two key advantages: it is weaker  than the PL condition, and it holds for a wide class of functions found in practical optimization problems \cite{karimi_linear_2016, yue_lower_2023, xiao_alternating_2023, foster_uniform_2018, li_calculus_2018}. Notably, recent studies have shown that the loss landscapes of many  neural networks   satisfy the KL condition in various settings \cite{bassily_exponential_2018, liu_loss_2022, scaman_convergence_2022}, highlighting its usefulness for analyzing convergence in non-convex optimization. 

%%%%%%%%%%%%%%%%%%%%%%%%%%%%%%%%%%%%
\subsection{Organization}
In Section \ref{extension of AEGD}, we introduce the gAEGD method and establish its unconditional energy stability. Section \ref{Convergence theory} presents the main theoretical results, including the two-stage adaptive behavior, robustness with respect to the base step size $\eta$, and convergence guarantees in non-convex settings. In Section \ref{convergence rates}, we analyze  convergence rates under a local KL condition. {Section \ref{experimental results} provides empirical evaluation of the proposed framework.} Conclusions and further discussions are given in Section \ref{conclusion}.  
%\red{We thank the referees for their constructive feedback and insightful comments, which have significantly improved the manuscript.}
%%%%%%%%%%%%%%%%%%%%%%%%%%%%%%%%%%%%%%%%%%5%%
\subsection{Notations}
Throughout the paper,  $\|\cdot\|$ denotes the $l_2$ norm for vectors and matrices. For the objective  function $f$, we write $f_k=f(x_k)$ and $\nabla f_k=\nabla f(x_k)$  to represent its value and gradient at iteration $k$, respectively. The global minimum value of $f$ is denoted by $f^*$.  We use $C^{1,1}_{L}(\mathbb{R}^d)$ to denote the class of functions with $L-$Lipschitz continuous gradients. That is, a function $f \in C^{1,1}_{L}(\mathbb{R^d})$ satisfies  
\begin{align*}
    \| \nabla f(x) - \nabla f(y) \| \leq L\|\ x-y \|,\ \forall x,y \in \mathbb{R}^d.
\end{align*}
This condition is equivalent to the following upper bound on $f$:  
\begin{align*}
    f(x) \leq f(y) + \nabla f(y)^T (y - x) + \frac{L}{2}\|y - x\|^2, \ \forall x,y \in \mathbb{R}^d.
\end{align*}

\section{Method} \label{extension of AEGD}
\subsection{Method derivation}
We begin by rewriting AEGD update  (\ref{AEGD}) as follows: 
\begin{align*}
    x_{k+1} &= x_{k} - 2\eta r_{k+1}\nabla F_k, 
\end{align*}
with the auxiliary update given by 
\begin{align*}
    r_{k+1}-r_k = -2\eta r_{k+1} \|\nabla F_k\|^2 = \nabla F_k \cdot (x_{k+1} - x_{k}). 
\end{align*} 
Recall that the deformed objective function is defined as $F(x) = \sqrt{f(x) + c}$. Therefore, the gradient of $F$ at step $k$ is 
$$
\nabla F_k = \frac{\nabla f(x_k)}{2\sqrt{f(x_k) + c}}=\frac{\nabla f(x_k)}{2F_k}. 
$$ 
Substituting this into the update rule, we obtain 
\begin{align*}
    x_{k+1} - x_{k}
    &= - \eta \frac{r_{k+1}}{F_k}\nabla f_k,
\end{align*}
and 
\begin{align*}
    r_{k+1} - r_{k}
    &= -\eta r_{k+1} \frac{F'_k}{F_k}\|\nabla f_k\|^2. 
\end{align*}
This leads to the following update rule:
\begin{subequations}
    \begin{align}
        x_{k+1} &= x_{k} - \eta \frac{r_{k+1}}{F_k}\nabla f_k, \\
        r_{k+1} &= \frac{r_k}{1+\eta\frac{F'_k}{F_k}\|\nabla f_k\|^2}.
    \end{align}
\label{updating rule}
\end{subequations}
%%%%%%%%%%%%%%%%%%%%%%%%%%%%%%
In this formulation, we consider the composite objective function defined as  $F(x)=\hat F(f(x)+c)$, where the function $\hat F: \mathbb{R}^+ \to \mathbb{R}^+$ generalizes the transformation used in AEGD. In the original AEGD method, a specific choice of $\hat F(s)=\sqrt{s}$ is used. Our goal is to extend this framework by allowing more general choices of $\hat F$, while maintaining the desirable convergence properties. Before discussing the structural conditions on 
$\hat F$, we clarify the notations used in (\ref{updating rule}): 
\begin{align}\label{ff} 
F_k: =\hat F(f(x_k)+c),\quad 
F_k':=\hat F'(f(x_k)+c).
\end{align} 
Given the function $\hat F$ and the notation introduced above, the update scheme in (\ref{updating rule}) gives a well-posed formulation. 

\subsection{Admissible energy functions} \label{admissible energy functions}
If $\hat F$ is strictly increasing,   energy stability is guaranteed regardless of the step size $\eta$.
This is formalized in the following result. 
\begin{theorem}[Energy stability]
Consider solving problem (\ref{target}). If $\hat F: \mathbb{R}^+ \to \mathbb{R}^+$ is strictly increasing, then the update rule (\ref{updating rule}) with (\ref{ff}) are unconditionally energy stable. Specifically, for any step size $\eta > 0$, 
the following identity holds: 
\begin{align}
    r_{k+1}^2 = r_k^2 - (r_{k+1} - r_k)^2 - \frac{2}{\eta}F_k F'_k \|x_{k+1} - x_k\|^2. 
\label{r_k^2 equality}
\end{align}
In particular, the sequence $\{r_k\}$ is strictly decreasing and converges to a non-negative limit $r^*$, i.e., $r_k \rightarrow r^*$,  as $k \rightarrow \infty$.
\label{unconditional energy stability}
\end{theorem}
\begin{proof}
    From the update equations in (\ref{updating rule}), we have 
    \begin{align}
        2r_{k+1}(r_{k+1}-r_k) &= -2\eta \frac{F_k^{'}}{F_k} r_{k+1}^2 \|\nabla f_k\|^2\nonumber \\
        &= -\frac{2}{\eta} F_k F_k^{'} \|x_{k+1}-x_k\|^2.
    \label{x_k_equality}
    \end{align}
    Then, the equation (\ref{r_k^2 equality}) can be obtained by rewriting the left hand side with $2a(a - b) = a^2 - b^2 + (a - b)^2$, {which indicates that $r_k^2$ is decreasing and bounded below, therefore convergent.} So does $r_k$ since $r_k \geq 0$.
\end{proof}
However, if $r_k$ decays too rapidly, the solution sequence may diverge. Therefore, to ensure convergence of the algorithm, we must further restrict the choice of $\hat{F}$.  In particular, it is  necessary to control the gap between $F_k$ and $r_k$ across  iterations. Define the energy gap as  
\begin{align*}
    e_k=F_k-r_k.
\end{align*}
As we will show next, a useful estimate for this gap can be obtained  under additional structural conditions on $\hat F$, particularly when $F$ is  concave. 

\vspace*{1\baselineskip}
\begin{theorem}[Gap between $r_k$ and $F_k$]
    Suppose $f \in C^{1,1}_{L}(\mathbb{R}^d)$ and is bounded  below by $f^*$. Let $\hat{F}: \mathbb{R}^+ \to \mathbb{R}^+$ be an increasing and concave function. Let $\{r_k\}$ and $\{F_k\}$ be the sequences generated by (\ref{updating rule}) with (\ref{ff}). Then, for any $k > m \geq 0$, the following estimate holds: 
    \begin{align}
        e_k:=F_k-r_k \leq e_m + \frac{L\eta}{4F^*}(r_m^2 - r_k^2),
    \label{distance between r and F}
    \end{align}
    where $F^* = \hat{F}(f^* + c)$. 
\label{theorem_distance between r and F}
\end{theorem}
This inequality quantifies how the gap $e_k$ evolves over iterations and shows that, under suitable assumptions, the gap remains controlled.

\begin{proof}
    Consider the update rule (\ref{updating rule}), since $f \in C^{1,1}_{L}(\mathbb{R}^d)$, we can apply the standard smoothness inequality. For all $j \geq 0$,
    \begin{align*}
        f_{j+1} 
        &\leq f_{j} + \nabla f_{j}^T(x_{j+1} - x_{j}) + \frac{L}{2}\|x_{j+1} - x_{j}\|^2.
    \end{align*}
    Using the updated expression $x_{j+1}- x_j=-\eta r_{j+1}/F_j \nabla f_j$, this becomes 
    \begin{align*} 
    f_{j+1} \leq  f_{j} - \eta \frac{r_{j+1}}{F_j} \| \nabla f_j \|^2 + \frac{L\eta^2 r_{j+1}^2}{2F_j^2} \| \nabla f_j \|^2.
    \end{align*}
    Next, recall that 
    $$
    r_{j+1}-r_j=-\eta \frac{F_j'}{F_j}
    r_{j+1}\|\nabla f_j\|^2.
    $$
    Rewriting the above inequality using this identity, we obtain:
    \begin{align*} 
        f_{j+1} \leq  f_j + \frac{1}{F'_j}(r_{j+1} - r_j) - \frac{L\eta}{2F_jF'_j}r_{j+1}(r_{j+1} - r_j).
        \end{align*} 
    Now applying the identity    $2a(a - b) = a^2 - b^2 + (a - b)^2$ to rewrite the last term: 
    $$
    r_{j+1}(r_{j+1}-r_j)=\frac{1}{2}
    (r_{j+1}^2-r_{j}^2 + (r_{j+1}-r_j)^2),
    $$    
    so the inequality becomes:
 \begin{align*}
     f_{j+1} & \leq  f_j + \frac{1}{F'_j}(r_{j+1} - r_j) + \frac{L\eta}{4F_jF'_j}\big[ r_{j}^2 - r_{j+1}^2 - (r_{j+1} - r_{j})^2 \big]\\
        &\leq f_j + \frac{1}{F'_j}(r_{j+1} - r_j) + \frac{L\eta}{4F_jF'_j}(r_{j}^2 - r_{j+1}^2).
    \end{align*}
   Since $\hat{F}$ is concave, we have  
   $$
   F_{j+1} - F_{j} \leq F'_j(f_{j+1} - f_j),
   $$
   which implies 
    $$ 
    e_{j+1}- e_j =F_{j+1} -F_j -(r_{j+1}-r_j) \leq F_j'(f_{j+1}-f_j)- (r_{j+1}-r_j) \leq 
    \frac{L\eta}{4F_j} (r_j^2 - r_{j+1}^2).
    $$
    Since $F_j \geq F^*:=\hat F (f^*+c)$, we can bound this as 
    $$
    e_{j+1}\leq e_j + \frac{L\eta}{4F^*} (r_j^2 - r_{j+1}^2).
    $$
    Summing both sides over $j=m$ to $k-1$, we obtain (\ref{distance between r and F}). 
\end{proof}
To further analyze the generalized  AEGD method we now state structural  assumptions on $\hat F$. 
\begin{assumption} 
    The function $\hat{F}: \mathbb{R}^+ \rightarrow \mathbb{R}^+$ explored in the update rule (\ref{updating rule}) is assumed to be smooth, strictly increasing such that $ \hat F' > 0$ and concave.
\label{assumptions on F}
\end{assumption}
These conditions are sufficient to control the disparity between $r_k$ and $F_k$, consequently leading to further nice properties of the update rule; see Section \ref{Convergence theory}. 

There exists a broad class of functions satisfying Assumption \ref{assumptions on F}. Through our numerical experiments, we observed that using $\hat{F}(s) = log(s + 1)$ often results in better empirical  performance compared to the classical choice $\hat{F}(s) = \sqrt{s}$ in AEGD. 

\section{Convergence theory: non-convex optimizaiton}\label{Convergence theory}
In this section, we present key  theoretical results concerning the update rule (\ref{updating rule}). 
We begin by deriving several upper and lower bounds that will be used in the subsequent analysis.

\begin{lemma}\label{important bounds}
    Suppose $f \in C^{1,1}_{L}(\mathbb{R}^d)$ and is bounded  below by $f^*$. Let the sequence $\{x_k\}_{k=0}^{\infty}$ be   generated by the update rule (\ref{updating rule}), where  $\hat F$ satisfies Assumption \ref{assumptions on F}. Then for any $k > 0$, the following bounds hold: 

    \quad(1) $F^* \leq F_k \leq \Bar{F}$, where $F^* = \hat F(f^*+c)$ and $\Bar{F}: = F_0 + \frac{L\eta r_0^2}{4F^*}$, 
     
    \quad(2) $f^* \leq f_k \leq \Bar{f}:= \hat F^{-1}(F_0 + \frac{L\eta r_0^2}{4F^*}) - c$,

    \quad(3) $\Bar{F'} \leq F'_k \leq F'^*$, where $\Bar{F'} = \hat F'(\Bar{f} + c)$ and $F'^* = \hat F'(f^* + c)$.
\end{lemma}

\begin{proof}
(1) By Theorem \ref{theorem_distance between r and F}, setting $m=0$ yields the following bound:
\begin{align}
    e_k \leq e_0 + \frac{L\eta}{4F^*}(r_0^2 - r_k^2), \ \forall k > 0, 
\label{distance estimation}
\end{align}
which leads to the upper bound for $F_k$ as follows:  
\begin{align}
    F_k =e_k+r_k \leq e_0  + \frac{L\eta r_0^2}{4F^*}+r_0 =\bar F
\label{upper bound of F_k}
\end{align}
valid for all $k>0$. The lower bound follows directly from the definition of $F_k$: 
$$
F_k=\hat F(f(x_k)+c)\geq \hat F(f^*+c)=F^*. 
$$
(2) and (3) follow from the monotonicity 
and concavity of $\hat F$, which imply that $f_k$ and $F_k'$ are bounded accordingly due to the inverse and derivative behavior of $\hat F$.  
\end{proof}
We now proceed to examine the convergence properties of the update rule ($\ref{updating rule}$). It can be equivalently rewritten as:
\begin{align}
    x_{k+1} = x_{k} - \eta_k\nabla f_k, \quad   \eta_k = \eta \frac{r_{k+1}}{F_k},
\label{rewrite of update rule}
\end{align}
and the energy parameter update is given by 
$$
r_{k+1} = \frac{r_k}{1+\eta\frac{F'_k}{F_k}\|\nabla f_k\|^2}.
$$
If $f \in C^{1,1}_{L}(\mathbb{R}^d)$, then for all $k\geq 0$, we have: 
\begin{align}
    f_{k+1} 
    &\leq f_k + \nabla f_k^T(x_{k+1} - x_{k}) + \frac{L}{2}\|x_{k+1} - x_{k}\|^2 \nonumber \\ 
    &= f_k - \eta_k\|\nabla f_k\|^2 + \frac{L\eta_k^2}{2}\|\nabla f_k\|^2 
    = f_k - \eta_k(1 - \frac{L\eta_k}{2})\|\nabla f_k\|^2.
\label{reformulated upper bound of f}
\end{align}
 Summing both sides from $k=0$ to $\infty$, we obtain 
\begin{align}\label{bounded infinity sum}
    \sum_{k=0}^{\infty}\eta_k(1 - \frac{L\eta_k}{2})\|\nabla f_k\|^2 \leq f_0 - f^*.  
\end{align}
This inequality implies that  $\|\nabla f_k\|^2 \to 0 \ \text{as} \ k\to \infty$, provided there exists an integer $N\in \mathbb{N}$ and constants $a,b>0$ such that 
$$ 
a \leq \eta_k \leq b < \frac{2}{L}, \ \forall \; k\geq N. 
$$
To show the upper bound on $\eta_k$ is not restrictive, we first establish the following lemma. 
\vspace{1\baselineskip}
\begin{lemma}[Two-stage behavior]
    Consider the update rule (\ref{updating rule}). For any $\epsilon > 0$ and $C>0$, define  
    $$
    N_0 = \left(\frac{ln(\frac{r_0}{C})}{ln(1+\underline{\alpha} \eta \epsilon)}\right)_+ := \max 
    \left\{
    \frac{ln(\frac{r_0}{C})}{ln(1+\underline{\alpha} \eta \epsilon)}, 1 \right\}, \; \text{where}\; \underline{\alpha}:= \frac{\Bar{F'}}{\Bar{F}}. 
    $$
    Then if $N \geq N_0$, at least one of the following holds: 
    $$
    \min\limits_{k\leq N}\| \nabla f_k \|^2 < \epsilon \quad  \text{or} \quad   r_{N} \leq C.
    $$
\label{two-stage}
\end{lemma}

\begin{proof}
By (1) and (3) in Lemma \ref{important bounds}, we have  
$$
\alpha_k = \frac{F'_k}{F_k}\geq \underline{\alpha}. 
$$
For any $N > 0$, we consider
$$
r_N = \frac{r_0}{\prod\limits_{k=0}^{N-1}(1 + \alpha_k \eta \| \nabla f_k \|^2)}.
$$
Using the lower bound $\alpha_k \geq \underline{\alpha}$, we get 
    \begin{align}
        r_N  \leq \frac{r_0}{(1 + \underline{\alpha}\eta \min\limits_{k\leq N-1}\| \nabla f_k \|^2)^N}.
    \label{r_N inequality}
    \end{align}
    Now fix any constant $C > 0$. Suppose  $r_N > C$ for some $N$, then by the above inequality,
    \begin{align*}
        \min\limits_{k \leq N-1} \| \nabla f_k \|^2 \leq \frac{(\frac{r_0}{r_N})^{\frac{1}{N}} - 1}{\underline{\alpha}\eta} < \frac{(\frac{r_0}{C})^{\frac{1}{N}} - 1}{\underline{\alpha}\eta}.
    \end{align*}
    Therefore, if 
    $$
    N \geq \frac{ln(\frac{r_0}{C})}{ln(1 + \underline{\alpha}\eta \epsilon)},
    $$
    then 
   $$
   \min\limits_{k \leq N-1} \| \nabla f_k \|^2 < \epsilon.
   $$
   Conversely, if $\min\limits_{k \leq N-1} \| \nabla f_k \|^2 \geq \epsilon$, then (\ref{r_N inequality}) implies 
   $$
   r_N \leq \frac{r_0}{(1 + \underline \alpha \eta \epsilon)^N} \leq C,
   $$
   provided $N \geq \frac{ln(\frac{r_0}{C})}{ln(1 + \underline \eta \epsilon)}$.
\end{proof}
It is worth emphasizing that the result provided in Lemma \ref{two-stage} holds unconditionally. According to this result,  for any $\epsilon > 0$,  after 
$$
N_0 = \left(\frac{ln(\frac{r_0}{C})}{ln(1+\underline{\alpha} \eta \epsilon)}\right)_+ := \max \left\{\frac{ln(\frac{r_0}{C})}{ln(1+\underline{\alpha} \eta \epsilon)}, 1\right\} \quad \text{with}\quad C = \frac{F^*}{L\eta}, 
$$ 
iterations,  we either reach an iterate $x_k$ such that $\| \nabla f_k \|^2 < \epsilon$, or 
$$
r_{N_0} \leq \frac{F^*}{L\eta},
$$
which implies that for all $k\geq N_0$, the effective step size satisfies  
$$
\eta_k = \eta\frac{r_{k+1}}{F_k} \leq \eta\frac{r_{N_0}}{F^*} \leq  \frac{1}{L}=: b.
$$
On the other hand, for all $k\geq 0$,  we also have the lower bound  
$$
\eta_k = \eta\frac{r_{k+1}}{F_k} \geq \eta\frac{r^*}{\Bar{F}},
$$
where $r^*:=inf_k r_k$ and $\Bar F:=\sup_k F_k$. 
Therefore, to ensure a strictly positive lower bound $a > 0$ on the step size, it is necessary  to establish conditions under which $r^*>0$. 

In the following, we derive sufficient conditions on the parameters, $\eta$, $r_0$ and $c$ to ensure $r^* > 0$. 

\begin{lemma}\label{conditions for r^*>0}
    Suppose $f \in C^{1,1}_{L}(\mathbb{R}^d)$ and is bounded below by $f^*$. Let $\{r_k\}_{k=0}^{\infty}$ be the sequence generated by (\ref{updating rule}), with $\hat{F}$ satisfying Assumption \ref{assumptions on F}. If $r_0 > F_0-F^*$, then 
    $$
    r_k \geq r^* >0 \quad \forall \; k\geq 0,
    $$
    provided $$
    \eta < \eta_{r_0} := \frac{4F^*(r_0-F_0+F^*)}{Lr_0^2}.
    $$
\end{lemma}
\begin{proof}
    The inequality (\ref{distance between r and F}) with $m = 0$ gives that for all $k > 0$,
    \begin{align}
        r_k \nonumber
        & =F_k-e_k\\ \nonumber
        & \geq F_k - \big( e_0  + \frac{L\eta }{4F^*} (r_0^2-r_k^2)\big)\\ \label{r^*}%\nonumber
        & \geq F^* + r_0 - F_0  - \frac{L\eta r_0^2}{4F^*} \\ \nonumber
        & = \frac{Lr_0^2}{4F^*}(\eta_{r_0} - \eta),
    \end{align}
    which implies  $r_k \geq r^*>0$ whenever  $r_0 > F_0 - F^*$ and $\eta < \eta_{r_0}$. 
\end{proof}

\begin{remark}
Let us consider the choice
$r_0=F_0$, which is the default setting for consistency with the derivation of the method. In this case,  the expression for $\eta_{r_0}$ simplifies to  
$$
\eta_{r_0}= \frac{4}{L}\left(\frac{F^*}{F_0}\right)^2 \sim \frac{4}{L} 
$$
when {$F^*\sim F_0$}, which holds for sufficiently large $c$. This asymptotic value clearly exceeds the classical upper threshold $2/L$, which is known to be necessary for ensuring the stability of the standard gradient descent update. 
\label{eta_r0}
\end{remark}
Next, we examine the role of the parameter $c$, as it appears in equation (\ref{ff}).  Given  any $r_0 > 0$ and $\eta > 0$, can we choose $c$ in such a way that the  conditions $$
r_0 > F_0 - F^* \; \text{and}\;  \eta < \eta_{r_0}
$$
are always satisfied? If so, this would  relax the conditions on both $r_0$ and $\eta$,
demonstrating the method's robustness. 

The following lemma establishes a lower bound on $c$ that ensure these conditions are met, thereby enhancing the algorithm's stability 
across a wide range of initializations and step sizes. 

\begin{lemma}
    Suppose $f \in C^{1,1}_{L}(\mathbb{R}^d)$ and is bounded  below by $f^*$. Let $\hat{F}: \mathbb{R}^+ \to \mathbb{R}^+$  be a smooth, strictly increasing, and strictly concave fucntion satisfying   
    $$
\lim\limits_{s\to \infty}\hat{F}(s) = \infty, \; \text{and} \;  \lim\limits_{s\to \infty}\hat{F}'(s) = 0.
$$
Let $\{r_k\}_{k=0}^{\infty}$ be the sequence generated by the update rule (\ref{updating rule}) with such a $\hat F$. Then for any fixed $\eta >0$ and $r_0 > 0$, if the parameter $c$ satisfies 
    \begin{align*}
        c \geq c^*, \quad c^*:=\min\limits_{a\in (0,1)} \max\left\{ \hat{F}'^{-1}\left(\frac{ar_0}{f_0 - f^*}\right),\ \hat{F}^{-1}\left(\frac{L\eta r_0}{4(1 - a)}\right) \right\} - f^*,
    \end{align*}
then $r_k \geq r^* > 0$ for all $k > 0$.
\label{condition on c for robust}
\end{lemma}
\begin{proof} 
    To ensure $r^*>0$ for any given $r_0 > 0$ and $\eta > 0$, we seek $c$ such that the conditions 
    $$
    r_0 > F_0 - F^* \; \text{and}\; \eta < \eta_{r_0},
    $$
    (from Lemma \ref{conditions for r^*>0}) hold. We achieve this in two steps.
    
   {\bf Step 1: Condition on $r_0 >F_0-F^*$}.\\
   Define  
    \begin{align*}
        G(c)= F_0 - F^*.  
    \end{align*}
    By the concavity of $\hat{F}$ and $f_0 > f^*$, we have 
    \begin{align*}
        G(c) & =  \hat{F}(f_0 + c) - \hat{F}(f^* + c) \\
        &  \leq \hat{F}'(f^* + c)(f_0 - f^*).
    \end{align*}
    Therefore, to ensure $r_0>G(c)$, it suffices to select  $c$ such that  
    \begin{align*}
        r_0  >  a r_0 \geq \hat{F}'(f^* + c)(f_0 - f^*)  \geq G(c)
    \end{align*}    
    for some $a\in (0, 1)$, where $a$ is to be determined later. Solving for $c$, we  require 
        \begin{align*}
            c \geq c_1(a):=  \hat{F}'^{-1}\left(\frac{ar_0}{f_0 - f^*}\right) - f^*.
        \end{align*}   
    {\bf Step 2: Condition on} $\eta <\eta_{r_0}$. \\   
    Recall that 
    $$
    \eta_{r_0}=\frac{4F^*(r_0-F_0+F^*)}{Lr_0^2}=\frac{4F^*(r_0-G(c))}{Lr_0^2}. 
    $$
    To ensure  $\eta < \eta_{r_0}$, it suffices to ensure   
        \begin{align*}
            \frac{1}{4} L \eta r_0^2 
            &<  \hat F(f^*+c)(r_0 - G(c)). 
        \end{align*}
    Note that under $c\geq c_1(a)$, we have $r_0-G(c) > (1-a)r_0$.    Thus to ensure  $\eta < \eta_{r_0}$, it suffices to also  select $c$ such that
    \begin{align*}
        \frac{1}{4}L\eta r_0^2 \leq  \hat{F}(f^* + c)(1 -a)r_0.
    \end{align*}
    Given that $\hat{F}$ is strictly increasing with $\lim\limits_{s\to\infty}\hat{F}(s) = \infty$, it is adequate to choose 
    \begin{align*}
        c \geq  c_2(a) := \hat{F}^{-1}\left(\frac{L\eta r_0}{4(1 - a)}\right) - f^*.
    \end{align*}
    Combining both steps, we conclude that $r_k \geq r^*>0$ for all $k\geq 0$ if  
    \begin{align*}
        c \geq  \max\{ c_1(a),\ c_2(a)\} 
    \end{align*}
    for  $a\in (0, 1)$.  It suffices to take $c\geq c^*$, provided  
    \begin{align*}
       c^* =\min\limits_{a\in (0,1)} \max\{c_1(a),\ c_2(a)\} 
    \end{align*}
    exists.     Now, we prove the existence of $c^*$. Note that:  
    \begin{align*}
        c_1'(a) = \frac{r_0}{(f_0 - f^*)\hat{F}''(c_1(a) + f^*)}, \\
        c_2'(a) = \frac{L\eta r_0}{4(1 - a)^2\hat{F}'(c_2(a) + f^*)}. 
    \end{align*}
    By the assumptions on $\hat{F}$, it's evident that $c_1(a)$ is decreasing, and $c_2(a)$ is increasing with respect to  $a$. The existence of $c^*$ will follow if we can demonstrate:   
    \begin{align*}
    \lim\limits_{a\to0^+}c_1(a) > c_2(0) \ \ \text{and} \ \lim\limits_{a\to1^-}c_2(a) > c_1(1).
    \end{align*}
From 
\begin{align*}
    \hat{F}'(c_1(a) + f^*) &= \frac{ar_0}{f_0 - f^*} \Rightarrow \hat{F}'(\lim\limits_{a\to0^+}c_1(a) + f^*) = 0,\\
    \hat{F}(c_2(a) + f^*) &= \frac{L\eta r_0}{4(1 - a)} \Rightarrow \hat{F}(\lim\limits_{a\to1^-}c_2(a) + f^*) = \infty,
\end{align*}
we conclude that 
\begin{align*} 
& \lim\limits_{a\to0^+}c_1(a) = \infty>c_2(0)= \hat{F}^{-1}(\frac{L\eta r_0}{4}) - f^* 
\quad \text{and} \\
& \lim\limits_{a\to1^-}c_2(a) = \infty > c_1(1)=\hat{F}'^{-1}\left(\frac{r_0}{f_0 - f^*}\right) - f^*.
\end{align*} 

\end{proof} 

We now state a non-asymptotic  convergence result of $\|\nabla f_k\|^2$. 
\begin{theorem}
Suppose $f \in C^{1,1}_{L}(\mathbb{R}^d)$ and is bounded below by $f^*$. Consider the update rule (\ref{updating rule}) where $F_k$ is defined via (\ref{ff}) using  $\hat{F}$ that is smooth, strictly increasing with 
$$
\lim\limits_{s\to \infty}\hat{F}(s) = \infty, 
$$
and strictly concave with $$
\lim\limits_{s\to \infty}\hat{F}'(s) = 0.
$$
Then for any $\eta > 0$ and $r_0 > 0$, if the parameter $c$ is chosen such that 
$$
c \geq c^*,  
$$
where $c^*$ is given in Lemma \ref{condition on c for robust}, 
then for any $\epsilon >0$, the gradient norm satisfies
$$
\min\limits_{k \leq N}\| \nabla f_k\|^2 < \epsilon
$$
after at most $N$ iterations with 
    \begin{subequations}
        \begin{align}
        N &= \Bigg\lceil \frac{1}{\epsilon} \Big[\frac{2\bar{F}(f_0 - f^*)}{\eta r^*}\Big] \Bigg \rceil \ \text{if} \ r_0 \leq \frac{F^*}{L\eta}, \\
        N &=  \Bigg\lceil \max \Big\{\frac{ln(\frac{Lr_0\eta}{F^*})}{ln(1 + \underline{\alpha} \eta \epsilon)}, 1 \Big\} + \frac{1}{\epsilon} \Big[\frac{2\bar{F}(\bar{f} - f^*)}{\eta r^*}\Big] \Bigg \rceil \ \text{if} \ r_0 > \frac{F^*}{L\eta}.
        \end{align}
    \label{N}
    \end{subequations} 
\label{non_asymptotic_convergence}
\end{theorem}

\begin{proof}
Under the stated assumptions and the condition $c\geq c^*$, Lemma \ref{condition on c for robust} guarantees that $r^* > 0$. This implies that for all $k\geq 0$, the effective step size satisfies  
    \begin{align}\label{lower_bound_of_etak}
        \eta_k = \eta\frac{r_{k+1}}{F_k} \geq \eta\frac{r^*}{\bar{F}},\quad \forall k > 0.
    \end{align}
    Now fix any $\epsilon>0$, by Lemma \ref{two-stage}, after 
    $$
    N_0 = \Big( \frac{ln(\frac{Lr_0\eta}{F^*})}{ln(1 + \underline \alpha \epsilon \eta)} \Big)_+
    $$ 
    iterations, if the condition $\min\limits_{k \leq N_0} \|\nabla f_k\|^2 < \epsilon$ is not met,  then it must hold that 
    $$
    r_{N_0} \leq \frac{F^*}{L\eta}.
    $$
    Let $k_0 \leq N_0$ be the smallest index such that $r_{k_0} \leq \frac{F^*}{L\eta}$.  
    Then for all $k\geq k_0$, it follows that 
    \begin{align*}
        r_k \leq \frac{F^*}{L\eta}, \; \text{and hence}\; \eta_k = \eta\frac{r_{k+1}}{F_k} \leq  \frac{1}{L}.
    \end{align*}
Combining this with the earlier bound \eqref{lower_bound_of_etak}, we obtain 
$$
\eta\frac{r^*}{\Bar{F}} \leq \eta_k \leq \frac{1}{L}, \quad \forall \; k\geq k_0.  
$$  
Given $c \geq c^*$, and recalling  the decent inequality (\ref{reformulated upper bound of f}), for any $j \geq k_0$,  we have 
    \begin{align*}
        f_{j+1} &\leq f_{j} - \eta_{j}(1 - \frac{L\eta_{j}}{2})\|\nabla f_{j}\|^2 \leq f_{j} - \frac{\eta r^*}{2\bar{F}}\|\nabla f_{j}\|^2,
    \end{align*}
    where the final inequality uses the bounds $\eta_{j} \geq  \eta r^*/\bar F$ and $\eta_{j} \leq 1/L$. Summing the inequality over $j=k_0$ to $k_0+M$, we obtain 
    \begin{align*}
        f_{k_0+M} - f_{k_0} 
        \leq -\frac{\eta r^*}{2\bar{F}}\sum\limits_{j=k_0}^{k_0+M}\| \nabla f_{j} \|^2
        \leq -\frac{\eta r^*}{2\bar{F}} M\min\limits_{j\leq k_0+M}\| \nabla f_{j} \|^2.
    \end{align*}
    Rearranging gives the bound
    \begin{align*}
        \min\limits_{j \leq k_0+M} \|\nabla f_{j}\|^2 
        \leq \frac{2\bar{F}(f_{k_0} - f_{k_0+M})}{M\eta r^*} 
        \leq \frac{2\bar{F}(f_{k_0} - f^*)}{M\eta r^*}. 
    \end{align*}
    {\bf Case 1:} $k_0=0$ \\
    Then it  suffices to choose 
    $$
    M > \frac{1}{\epsilon} \Big[ \frac{2\bar{F}(f_{0} - f^*)}{\eta r^*} \Big],
    $$
    to ensure 
    $\min\limits_{j\leq k_0+M} \|\nabla f_{j}\|^2 < \epsilon$, as claimed in (\ref{N}a).  

    {\bf Case 2:} $k_0>0$ \\
    In this case, we can use the  bound $f_{k_0}\leq \bar f =\sup_k f_k$,
    giving 
    \begin{align*}
        \min\limits_{j \leq k_0+M} \|\nabla f_{j}\|^2 
        \leq \frac{2\bar{F}(f_{k_0} - f^*)}{M\eta r^*}
        \leq \frac{2\bar{F}(\bar{f} - f^*)}{M\eta r^*}.
    \end{align*}
    Since $k_0\leq N_0$, this implies 
    \begin{align*}
        \min\limits_{j \leq N_0+M} \|\nabla f_{j}\|^2 < \epsilon,
    \end{align*}
    as long as 
    $$
    M > \frac{1}{\epsilon}\Big[ \frac{2\bar{F}(\bar{f} - f^*)}{\eta r^*} \Big],
    $$
    which is the result stated in (\ref{N}b).
\end{proof}

\section{Convergence rates}\label{convergence rates}
Given that  $\|\nabla f_k\| \to 0$ as $k\to \infty$, it is natural to investigate the convergence behavior of $\{f(x_k)\}$ and $\{x_k\}$, as well as their respective rates of convergence. To this end,  we introduce a structural condition on the objective function $f$ that facilitates such analysis.

Let $S$ denote the set of all saddle points and local minimizers of $f$. For given  parameters $\mu,\ \delta > 0$ and $\alpha \in (0,2)$, a differentiable function $f:\mathbb{R}^d \to \mathbb{R}$ is said to satisfy the $(\mu, \alpha, \delta)$-KL condition if for every $\tilde x \in S$, and for all $x$ in a neighborhood of $\tilde x$,
the following inequality holds:
\begin{align}\label{KL_condition}
    \|\nabla f(x)\|^2 \geq 2\mu (f(x) - f(\tilde{x}))^{\alpha},
\end{align}
 where all points $x$ in the neighborhood of $\tilde x$ satisfy 
$$
f(\tilde x) \leq f(x) \leq f(\tilde x) +\delta.
$$
Under this condition, the convergence behavior and rate of the sequence $\{x_k\}$ are characterized by the following theorem.

\begin{theorem}(Convergence under the KL condition) Suppose $f \in C^{1,1}_{L}(\mathbb{R}^d)$ is bounded below by $f^*$, and satisfies the $(\mu, \alpha, \delta)$-KL condition (\ref{KL_condition}) for some fixed $\mu, \delta > 0$ and $\alpha \in (0, 2)$. Consider the update rule (\ref{updating rule}) with $\hat{F}: \mathbb{R}^+ \to \mathbb{R}^+$ satisfying the same assumptions as in Theorem \ref{non_asymptotic_convergence}. If the parameter $c$ is chosen such that 
$$
c \geq \max\big\{c^*,\ \Tilde{c} := \Hat{F}^{-1}(L\eta r_0) - f^*\big\},
$$
then for any $\eta > 0$ and $r_0 >0$, the following hold:   

(i) There exists $\tilde x \in S$ and index  $N_1$ (depending on $\delta$) such that
$$
f(\tilde x) \leq f(x_k) \leq f(\tilde x) +\delta, \quad \forall k\geq N_1.
$$
(ii) The sequence $x_k$ converges to $\tilde x$, and the following convergence rates hold: 

(1) If $\alpha = 1$, then 
$$
\|x_k-\tilde x\| \leq C_1 e^{-Qk/2}, \; \forall  \; k\geq N_1.
$$

(2) If $\alpha \in (0, 1)$, then 
$\{x_k\}$ converges to $\tilde x$  in a finite number of steps. 

(3) If $\alpha \in (1, 2)$, then 
$$
\|x_k-\tilde x\| \leq C_2 (C_3+k)^{- \frac{2-\alpha}{2(\alpha-1)}}, 
\quad \forall k\geq N_1.
$$
Here, the constants are defined as follows:
\begin{align*} 
Q & = \frac{\mu\eta r^*}{F_0},\\
C_1 & = 2 \sqrt{\frac{2exp\{QN_1\}w_{N_1}}{\mu}},\\
C_2 & = \sqrt{\frac{2}{\mu}}\Big(\frac{2}{2-\alpha}\Big)\Big[(\alpha-1)Q\Big]^{-\frac{2-\alpha}{2(\alpha - 1)}},\\
C_3 & =\frac{w_{N_1}^{1 - \alpha}}{(\alpha-1)Q}-N_1,
\end{align*} 
where $w_k=f(x_k)-f(\tilde x)$.
\label{converge_rate}
\end{theorem}

\begin{proof}
We begin by showing that under the given conditions on $c$, the sequence $\{x_k\}$ will enter the region where the KL condition holds and remains in this region thereafter. 
    
Recall the reformulated form of the update rule \eqref{updating rule}:
    \begin{align}
        x_{k+1} = x_{k} - \eta_k\nabla f(x_k),\label{reformulation_1}
    \end{align}
    where the effective step size is defined as
    $$
    \eta_k = \eta\frac{r_{k+1}}{F_k}, \quad \text{with} \quad  r_{k+1} = \frac{r_k}{1+\eta\frac{F'_k}{F_k}\|\nabla f_k\|^2}.
    $$ 
   Since $f\in C_{L}^{1,1}(\mathbb{R}^d)$,
   we can apply the standard descent property of smooth functions to obtain the inequality
\begin{align}\label{GD_fk_scheme}
        f(x_{k+1}) \leq f(x_k) - \eta_k(1 - \frac{L\eta_k}{2})\|\nabla f(x_k)\|^2, \; \forall k\geq 0.
    \end{align}
Now, since  $c \geq c^*$, Lemma \ref{condition on c for robust} guarantees that  $r_k \geq r^* > 0$ for all $k$.  Furthermore,  the condition  $c\geq \tilde c := \hat{F}^{-1}(L\eta r_0) - f^*$ implies  
    $$
    L\eta r_0 \leq \hat F(f^*+c)=:F^*, 
    $$
    so that we obtain the following bounds on the effective step size:
    $$
    0<\eta\frac{r^*}{F_k} \leq \eta_k = \eta\frac{r_{k+1}}{F_k} \leq \eta \frac{r_0}{F^*}\leq \frac{1}{L}.
    $$
This ensures that $\{f_k\}$ is decreasing, and since  $F_k=\hat F(f_k+c)$, it follows that  
$$
F^* \leq F_k: = \Hat{F}(f_k + c) \leq \Hat{F}(f_0 + c) = F_0.
$$  
In particular, using $\eta_k \geq \eta r^*/F_0$, we get a uniform descent estimate:
\begin{align}\label{GD_scheme_etak_upper}
        f(x_{k+1}) \leq f(x_k) - \frac{\eta r^*}{2F_0} \|\nabla f(x_k)\|^2.
    \end{align}
This implies that  $f(x_k) \downarrow \tilde f$ for some limit $\tilde f \geq f^*$, and 
$\|\nabla f(x_k)\| \to 0$ 
   as  $k\to \infty$. By the update formula \eqref{reformulation_1}, it then follows that 
    $
    \|x_{k+1}-x_k\|\to 0.
    $
   
   (i) {\bf Identification of the limit point} $\tilde x$: \\
   Since $f\in C_L^{1, 1}$, the limit  $\tilde f =\lim f(x_k) \geq f^*$ is attained at some point $\tilde x$, with 
    $\tilde f=f(\tilde x)$ and 
    $
    \nabla f(\tilde x)=0. 
    $
    Therefore, $\tilde x \in S$, the set of saddle points and local minimizers. Moreover, $\tilde x$ cannot be a local maximum, since the sequence $\{f(x_k)\}$ is decreasing. 
    For any $\delta>0$, there exists $N_1$, which depends on $\delta$, such that 
    $$
    \tilde f \leq f(x_k)\leq \tilde f +\delta, \quad \forall \; k\geq N_1, 
    $$
    and hence $x_k$ remains in the neighborhood where the KL condition holds. 

    (ii) {\bf Convergence of} $\{x_k\}$. \\ 
Since $\eta_k \leq \frac{1}{L}$, inequality \eqref{GD_fk_scheme} gives
    $$
    f_{k+1} \leq f_k - \frac{\eta_k}{2}\|\nabla f_k\|^2 = f_k - \frac{1}{2}\|x_{k+1} - x_{k}\| \|\nabla f_k\|,
    $$
    where the equality follows from \eqref{reformulation_1}, $x_{k+1}-x_k=-\eta_k \nabla f_k$. 
    Define $w_k:= f(x_k) - f(\tilde x)$. For  all $k\geq N_1$, since $x_k$ lies in the KL region, the $(\mu, \delta, \alpha)$ KL condition gives: 
    \[
    w_{k} - w_{k+1} \geq \sqrt{\frac{\mu}{2}}w_k^{\frac{\alpha}{2}}\Vert x_{k+1} - x_k \Vert.
    \]
    Rearranging, we get 
      \[
   \Vert x_{k+1} - x_k \Vert \le  \sqrt{\frac{2}{\mu}}(w_k - w_{k+1})w_k^{-\frac{\alpha}{2}}.
    \]
    Since $s^{-\frac{\alpha}{2}}$ with $s>0$ is decreasing and $w_{k+1} \leq w_k$, {we bound the right hand side by an integral:}
    \[
    (w_k - w_{k+1})w_k^{-\frac{\alpha}{2}} \leq \int^{w_k}_{w_{k+1}}s^{-\frac{\alpha}{2}}ds = \frac{2}{2 - \alpha}(w_k^\frac{2 - \alpha}{2} - w_{k+1}^\frac{2 - \alpha}{2}).
    \]
   This leads to 
    \[
    \Vert x_{k+1} - x_k \Vert \le \sqrt{\frac{2}{\mu}}(\frac{2}{2 - \alpha})(w_k^\frac{2 - \alpha}{2} - w_{k+1}^\frac{2 - \alpha}{2}).
    \]
    Summing over $k$ from $N_1$ to $\infty$, we obtain 
    \[
    \sum_{k=N_1}^\infty \Vert x_{k+1} - x_k \Vert \leq \sqrt{\frac{2}{\mu}}(\frac{2}{2 - \alpha})w_{N_1}^\frac{2 - \alpha}{2}, \quad \text{for}\; \alpha \in (0,2).
    \]
    This confirms that $\{x_k\}$ is a Cauchy sequence and hence converges.  Moreover, for all $k \geq N_1$, we have 
    \begin{align} \label{d_x_k}
    \|x_k-\tilde x \| 
    & \leq \sum_{j=k}^\infty \|x_j-x_{j+1}\| \nonumber \\
    & \leq 
    \sqrt{\frac{2}{\mu}}(\frac{2}{2 - \alpha})w_{k}^\frac{2 - \alpha}{2}. 
    \end{align}
    We now estimate the convergence rate of $x_k$ via $w_k:=f(x_k)-f(\tilde x)$.
    
    Applying the KL inequality (\ref{KL_condition}) and using (\ref{GD_scheme_etak_upper}), define the constant $Q:= \mu\eta\frac{r^*}{F_0}$, then  
    \begin{align}
    w_{k} - w_{k+1} \ge \eta\frac{r^*}{2F_0} \|\nabla f_k\|^2 \ge Q w_k^\alpha, \quad  \forall \; k \geq N_1. 
    \label{15}
    \end{align}
We now consider three cases for the KL exponent $\alpha \in (0, 2)$: 
    
    (1) {\bf Case  $\alpha = 1$:}\\
    Inequality $\eqref{15}$ becomes 
    \[
    w_{k+1} \leq (1 - Q) w_k, \quad \forall k \geq N_1.
    \]
    By induction, this gives 
    \[
    w_{k} \leq (1 - Q)^{(k-N_1)} w_{N_1} = exp\{(k - N_1) log(1 - Q)\} w_{N_1} \leq \exp \{- (k-N_1)Q \}w_{N_1}.
    \]
  Applying this in \eqref{d_x_k}, we obtain
  exponential convergence of the iterates: 
    $$
    \|x_k-\tilde x\| \leq C_1 e^{-Qk/2}, \quad \forall k\geq N_1, 
    $$
  where 
    $$
    C_1= 2 \sqrt{\frac{2exp\{QN_1\}w_{N_1}}{\mu}} . 
    $$
    
    (2) {\bf Case $\alpha \in (0, 1)$:}\\
    Using the concavity of  $s^{1 - \alpha}$ for $s>0$, and inequality \eqref{15},  we have
    \[
    w_{k+1}^{1 - \alpha} - w_{k}^{1 - \alpha}  \leq (1 - \alpha)w_k^{- \alpha} (w_{k+1} - w_{k})\le -(1 - \alpha)Q.
    \]
    By induction, this leads to 
    \[
    w_{k}^{1 - \alpha} \leq w_{N_1}^{1 - \alpha} - (1- \alpha)(k-N_1)Q.
    \]
    The right hand side becomes negative when 
    $$
    k = k_0 := N_1 + \lceil \frac{w_{N_1}^{1-\alpha}}{(1 - \alpha)Q} \rceil,
    $$
    which is impossible, so $w_k$ must reach zero in finite steps.  Therefore, $x_k \to \tilde x$ in finite number of steps.       

    (3) {\bf Case $\alpha \in (1,2)$:}\\
    Using the concavity of the function $\frac{s^{1 - \alpha}}{1 - \alpha}$ for $s > 0$, and inequality \eqref{15}, we get  
    \[
    \frac{w_{k+1}^{1 - \alpha}}{1 - \alpha} - \frac{w_{k}^{1 - \alpha}}{1 - \alpha} \le  w_k^{-\alpha}(w_{k+1} - w_k) \le - Q, \quad \forall k \geq N_1.
    \]
    By induction, this implies:
    \[
    \frac{w_{k}^{1 - \alpha}}{1 - \alpha} \le \frac{w_{N_1}^{1 - \alpha}}{1 - \alpha} - (k-N_1)Q, \quad \forall k \geq N_1,
    \]
    which rearranges to 
    \begin{align*}
    w_{k} &\le [w_{N_1}^{1 - \alpha} + (\alpha -1)(k-N_1)Q]^\frac{1}{1-\alpha},
    \quad \forall \; k\geq N_1.
    \end{align*}
     Applying this bound to inequality \eqref{d_x_k} we obtain the convergence rate: 
    \begin{align*}  
    \|x_k-\tilde x\| & \leq 
    \sqrt{\frac{2}{\mu}}(\frac{2}{2 - \alpha})w_{k}^\frac{2 - \alpha}{2} \\
    & \leq \sqrt{\frac{2}{\mu}}(\frac{2}{2 - \alpha})
    [w_{N_1}^{1 - \alpha} + (\alpha -1)(k-N_1)Q]^{- \frac{2-\alpha}{2(\alpha-1)}}. 
    \end{align*}   
    This simplifies to the final convergence rate: 
    $$
     \|x_k-\tilde x\|  \leq  C_2 (C_3+k)^{- \frac{2-\alpha}{2(\alpha-1)}}, \; \forall k\geq N_1, 
    $$
   where the constants are defined as 
    $$
    C_2 = \sqrt{\frac{2}{\mu}}\Big(\frac{2}{2-\alpha}\Big)\Big[(\alpha-1)Q\Big]^{-\frac{2-\alpha}{2(\alpha - 1)}}, \quad C_3=\frac{w_{N_1}^{1 - \alpha}}{(\alpha-1)Q}-N_1.
    $$
\end{proof}
\begin{remark}
It is important to emphasize that for any given $\delta>0$, there exists a finite index  $N_1$ such that $f(x_k)\leq f(\tilde x)+\delta$ for all $k\geq N_1$. While an explicit formula for $N_1=N_1(\delta)$ is generally unavailable, a concrete instance occurs when $\delta = f(x_0)-f(\tilde x)$, in which case we can take $N_1 = 0$. 
    \iffalse 
    (2) When $\alpha = 1$, the convergence rate is consistent with those derived for functions that satisfy the PL condition, i.e., (17) with $\alpha =1$; as discussed in \cite{AEGD} and \cite{PL_2016}
    \fi  
\end{remark}
\begin{remark} 
%\red{ 
The convergence-rate analysis remains valid for all values of $\alpha$. However, if $f\in C^{1, 1}_{L}(\mathbb{R}^d)$,  then the KL condition (\ref{KL_condition}) excludes the possibility $\alpha \in (0, 1)$. 
To see this, let  $y=x-\frac{1}{L}\nabla f(x)$. Applying the descent lemma, we obtain 
$$
f(y)\leq f(x)-\frac{1}{2L}\|\nabla f(x)\|^2.  
$$ 
Hence 
$$
\|\nabla f(x)\|^2 \leq 2L(f(x)-f(y)) \leq 2L(f(x)-f(\tilde x)).
$$
Combing this with the KL inequality (\ref{KL_condition})  yields  
$$
\mu \le L w^{1-\alpha}, \quad  w:=f(x)-f(\tilde x).
$$
As $x\to \tilde x$, we have $w\to 0$. If $\alpha<1$, then $w^{1-\alpha}\to 0$, and the above inequality cannot hold for any fixed $\mu>0$. Therefore, under the L-smoothness assumption, the KL inequality  (\ref{KL_condition})  can hold only if $\alpha \geq 1$. %}
\end{remark}

\section{Experimental Results}\label{experimental results}

%\red{
In this section, we present experiments to validate  the theoretical results and assess the practical performance of the proposed gAEGD method. The section is organized into two parts: the first empirically examines the theoretical properties, while the second compares gAEGD with standard optimization methods including gradient descent with line search, gradient descent with momentum (GDM), conjugate gradient method, and Adam.
%}

\subsection{Validation of Theoretical Findings}

%\red{
This subsection validates the main theoretical results on two classes of problems: convex and non-convex objectives.
%}

For the convex objective, we consider a 100-dimensional quadratic function defined as: 
\begin{equation*}
f(x_1,x_2,...,x_{100})=\sum\limits_{i=1}^{50}x_{2i-1}^2+\sum\limits_{i=1}^{50}\frac{x_{2i}^2}{10^2}.
%\label{Quadratic}
\end{equation*}
This quadratic function is  strongly convex with a Lipschitz constant $L=2$. It has a global minimum  at $x^*=(0,0,...,0)$, where $f(x^*)=0$.

For the non-convex objective, we consider the 2D-Rosenbrock, defined as: 
\begin{align*}
    f(x)=(1 - x_1)^2 + b(x_2 - x_1^2)^2,
\end{align*}
where $b>0$ is parameter that controls the condition number of the Hessian $\nabla^2 f$. By default, we set  $b=100$, but  varying $b$ allows us to adjust the difficulty of the problem,  enabling a more comprehensive  evaluation of each algorithm's robustness. As $b$ increases, 
the condition number worsens, making optimization more challenging. 

This function has a global minimum  at $x^* = (1, 1)$, where  $f(x^*) = 0$.  The minimum lies within a long, narrow, parabolic shaped valley, which poses significant difficulty for optimization algorithms in terms of convergence speed and stability.  

For all experiments, we initialize $r_0 = F_0$. The base step size $\eta$ (referred to as ``$lr$'') is fine-tuned individually  for each algorithm to ensure optimal performance. Unless otherwise noted, the parameter $c$ is set to $1$.

%\red{
In Section \ref{numerical_admissible_energy_funtions}, we examine several examples of admissible energy functions introduced in Section \ref{admissible energy functions}. In Section \ref{numerical_two_stage}, we empirically validate the two-stage behavior established in Lemma~\ref{two-stage} across a range of tasks. Finally, in Section \ref{numerical_c}, we study the effect of the parameter $c$ on the algorithm's robustness with respect to the base step size  $\eta$.
%}

\subsubsection{Admissible Energy Functions} \label{numerical_admissible_energy_funtions}
As discussed in Section \ref{admissible energy functions}, the energy functions admissible under the update rule \eqref{updating rule} must satisfy the conditions outlined  in Assumption \ref{assumptions on F} -- namely, they must be  smooth, strictly increasing, and concave. A variety of functions meet these criteria. For clarity and simplicity, we focus on and discuss the following subset of such functions in this section: 
\begin{itemize}
    \item[(1)] Power energy functions, $\hat{F}(s) = s^p$ with $p = 0.1, 0.2, ..., 1.0$, 
    \item[(2)] Logarithmic energy function, $\hat{F}(s) = log(s + 1)$.
\end{itemize}
Notably, when $\hat{F}(s) = s^{0.5}$, the update rule \eqref{updating rule} reduces to the original AEGD method. When using the logarithmic energy function $\hat{F}(s) = log(s + 1)$, the corresponding algorithm is referred to as the \textit{Adaptive Logarithmic Energy Gradient Descent (ALEGD)} method.

We compare and analyze the performance of various algorithms associated with the energy functions described above, with particular focus on the comparison between AEGD and ALEGD. Our results show the advantages of ALEGD over AEGD, especially in non-convex scenarios.
\label{Performance comparison}
\begin{figure}[H]
    \centering
    \subfigure[]{
    \begin{minipage}[t]{0.5\linewidth}
    \centering
    \includegraphics[width=2.5in]{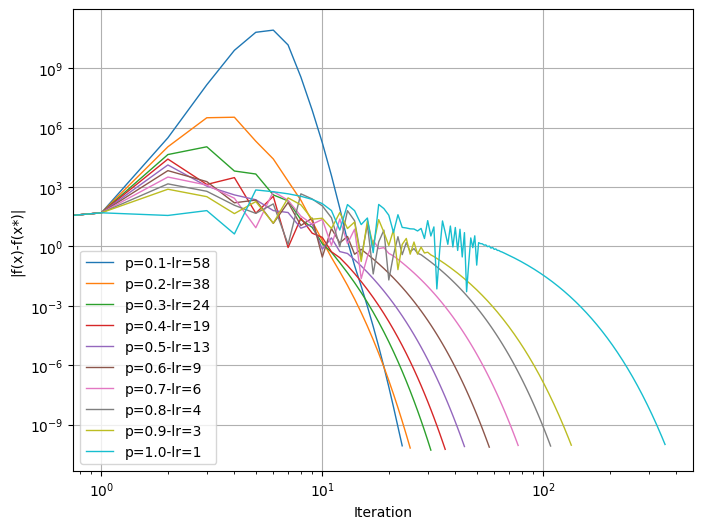}
    \label{quad_loss_a}
    \end{minipage}%
    }%
    \subfigure[]{
    \begin{minipage}[t]{0.5\linewidth}
    \centering
    \includegraphics[width=2.5in]{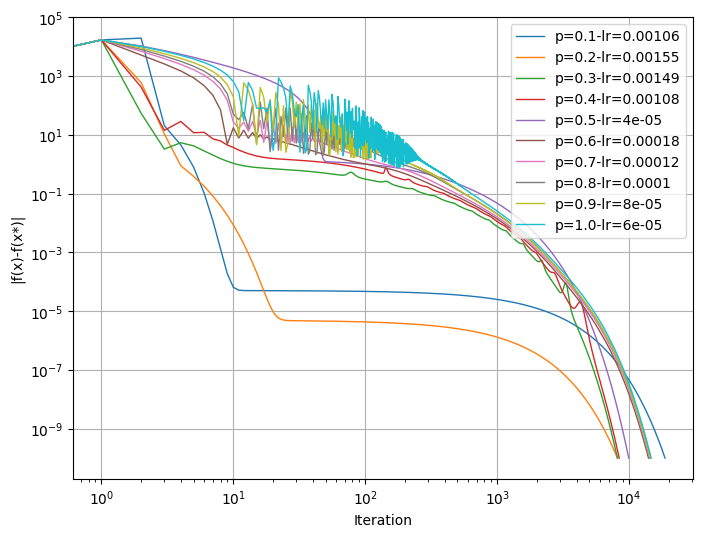}
    \label{quad_loss_b}
    \end{minipage}%
    }%
    \centering
    \caption{Loss comparison of using the update rule \eqref{updating rule} with different power energy functions $\hat{F}(s) = s^p$, where $p = 0.1, 0.2, ..., 1.0$, applied to (a) the 100D quadratic objective,  and (b) the 2D Rosenbrock objective.}
\label{power_energy}
\end{figure}

Figure \ref{power_energy} compares the performance of the update rule defined in \eqref{updating rule} using  various power energy functions $\hat{F}(s) = s^p$. The initial point is set to $(1, 1, ..., 1)$
for the 100-dimensional quadratic objective, and $(-3, -4)$ for the 2D Rosenbrock objective. For the 100D quadratic function, the results reveal a clear trend: as the exponent $p$ decreases from $1.0$ to $0.1$, the optimal base step size increases, and the number of iterations required to reach an accuracy of $10^{-10}$ decreases. This suggests improved convergence behavior for smaller values of $p$ in convex settings.
In contrast, no consistent pattern emerges for the 2D Rosenbrock function. This lack of regularity suggests that a universal trend for the power energy functions does not hold in a non-convex case, likely due to the more complex landscape of the objective.

\begin{figure}[h]
    \centering
    \includegraphics[width=0.45\linewidth]{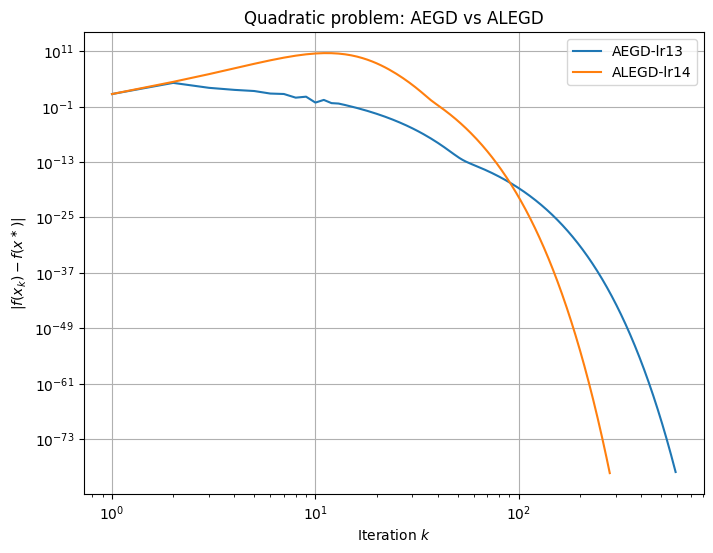}
    \caption{Loss comparison of AEGD and ALEGD on the 100D quadratic objective.}
    \label{qual_alegd_loss}
\end{figure}

Figure \ref{qual_alegd_loss} provides a direct comparison between AEGD and ALEGD on the 100D quadratic objective. The results show that ALEGD outperforms AEGD in the long run, particularly beyond about $90$ iterations or once the accuracy approaches approximately $10^{-18}$. This suggests that ALEGD may offer better long-term convergence properties in convex settings.

\begin{figure}[h]
    \centering
    \includegraphics[width=1\linewidth]{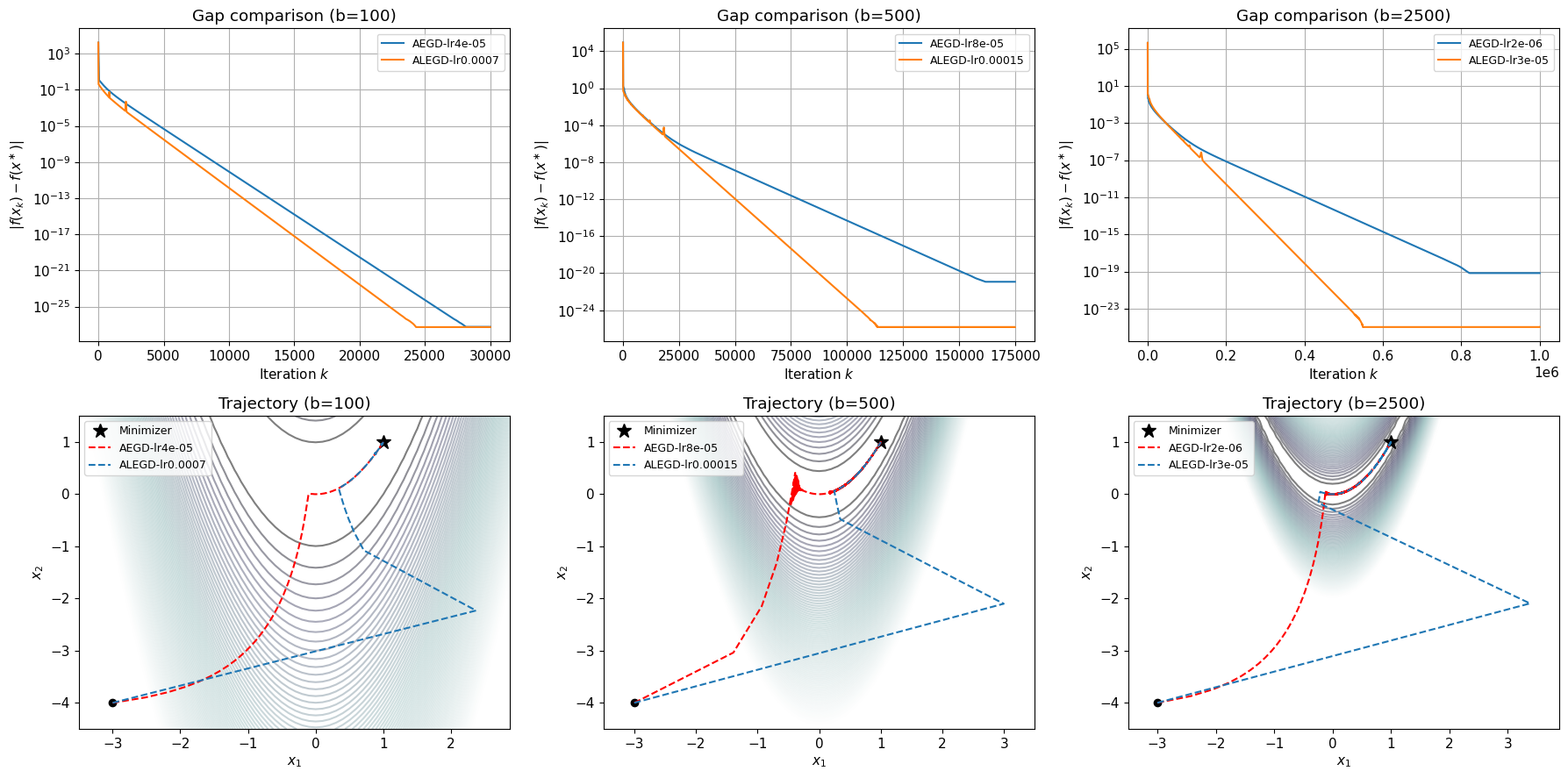}
    \caption{Loss and trajectory comparison of AEGD and ALEGD on the 2D Rosenbrock objective with varying levels of difficulty.}
    \label{rosen_alegd}
\end{figure}

\begin{table}[ht]
\centering
\caption{Average computation time (in seconds, averaged over $10$ runs) required by AEGD and ALEGD to reach an accuracy of $10^{-10}$ on the 2D Rosenbrock problem for different values of $b$.}
\begin{tabular}{lcc}
\toprule
 & AEGD & ALEGD \\
\midrule
$b = 100$  & 0.1517 & 0.1194 \\
$b = 500$  & 0.9452 & 0.5968 \\
$b = 2500$ & 5.6014 & 3.0933 \\
\bottomrule
\end{tabular}
\label{computation_time}
\end{table}

Figure \ref{rosen_alegd} compares the loss curves of AEGD and ALEGD when applied to the 2D Rosenbrock objective function, and displays  the corresponding optimization trajectories. Table \ref{computation_time} summarizes the average computation time required by each algorithm to reach a target accuracy of $10^{-10}$. 

%\red{
The results show that ALEGD consistently outperforms AEGD in both iteration count and computational time required to reach a desired accuracy, even under high condition number settings. In particular, ALEGD converges more rapidly than AEGD and attains a more accurate final solution in this task.
%}

%\red{
Although these results are obtained on a specific test problem and may not capture all aspects of performance, they  suggest that ALEGD -- based on the logarithmic energy function $\hat{F}(s) = log(s+1)$ -- is a competitive and effective alternative to AEGD, which employs the square root energy function $\hat F(s)=s^{0.5}$.
%}  

%\red{
Motivated by these observations, we focus on AEGD and ALEGD in the following two subsections to examine the influence of key parameters. Furthermore, in Section~\ref{experiments}, we use AEGD and ALEGD as representative instances of the proposed framework to evaluate overall performance.
%}

\subsubsection{Two-stage behavior} \label{numerical_two_stage}
Figure \ref{quad_two_stage} and  \ref{rosen_two_stage} compare the behavior of the energy parameter $r_k$ and the effective step size  $\eta_k$ for AEGD and ALEGD when applied to different objective functions. 

For the energy parameter $r_k$,
both AEGD and ALEGD exhibit similar convergence patterns in shape; however, they differ in convergence speed and scale.

Regarding the effective step size  $\eta_k$, ALEGD consistently starts with a higher initial value compared to AEGD across both problem types. In the quadratic case, $\eta_k$ for ALEGD converges to a value similar to that of AEGD.  However, in the Rosenbrock case, ALEGD converges to a lower final step size, possibly contributing to its improved stability in non-convex settings.

\begin{figure}[h]
    \centering
    \subfigure[]{
    \begin{minipage}[t]{0.5\linewidth}
    \centering
    \includegraphics[width=2.5in]{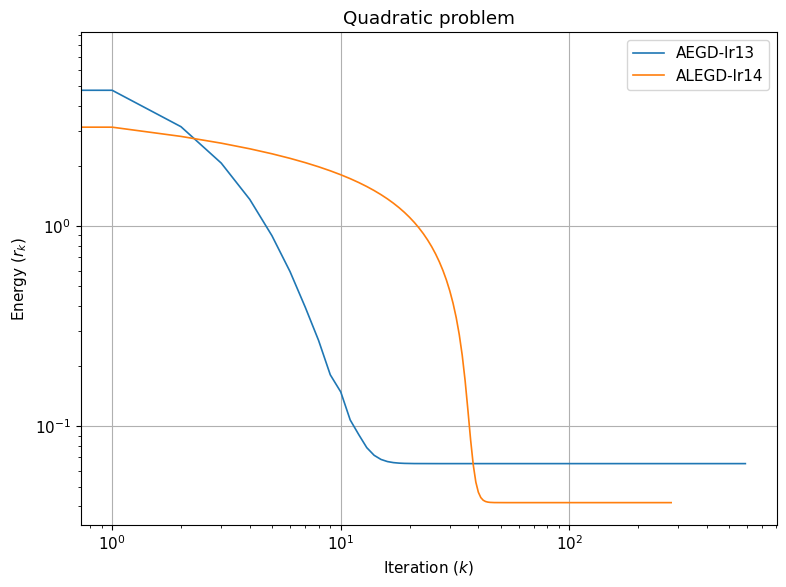}
    %\caption{}
    \end{minipage}%
    }%
    \subfigure[]{
    \begin{minipage}[t]{0.5\linewidth}
    \centering
    \includegraphics[width=2.5in]{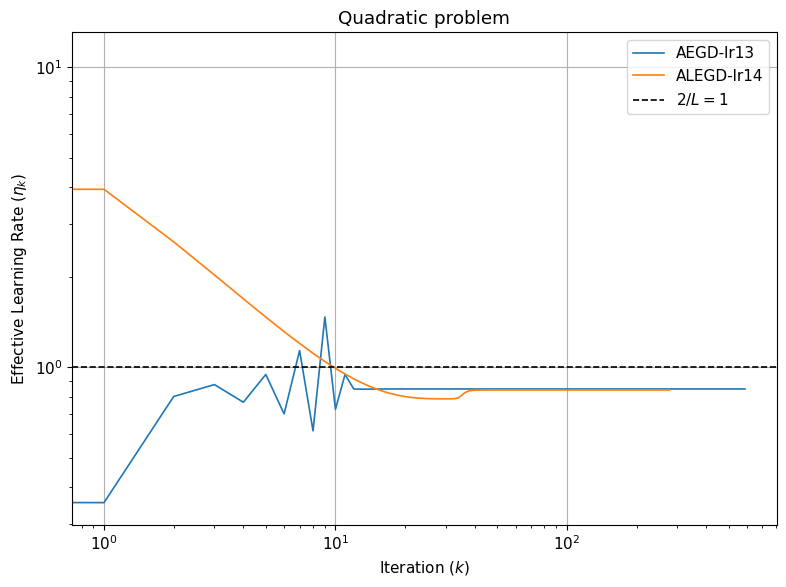}
    %\caption{}
    \label{quad_two_stage_b}
    \end{minipage}%
    }%
    \centering
    \caption{Comparison of the behavior of the energy parameter ($r_k$) and the effective step size ($\eta_k$) for AEGD and ALEGD when applied to the 100D quadratic objective.}
    \label{quad_two_stage}
\end{figure}

\begin{figure}[h]
    \centering
    \subfigure{
    \begin{minipage}[t]{0.5\linewidth}
    \centering
    \includegraphics[width=2.5in]{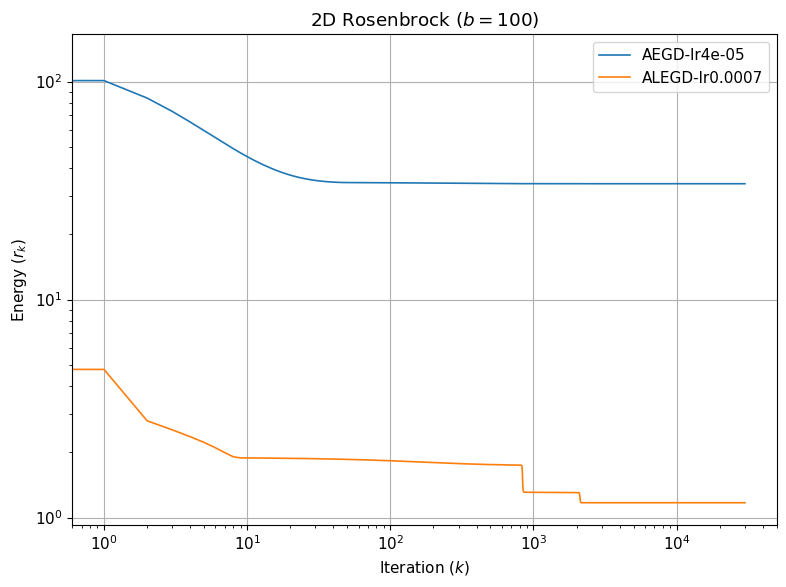}
    %\caption{}
    \end{minipage}%
    }%
    \subfigure{
    \begin{minipage}[t]{0.5\linewidth}
    \centering
    \includegraphics[width=2.5in]{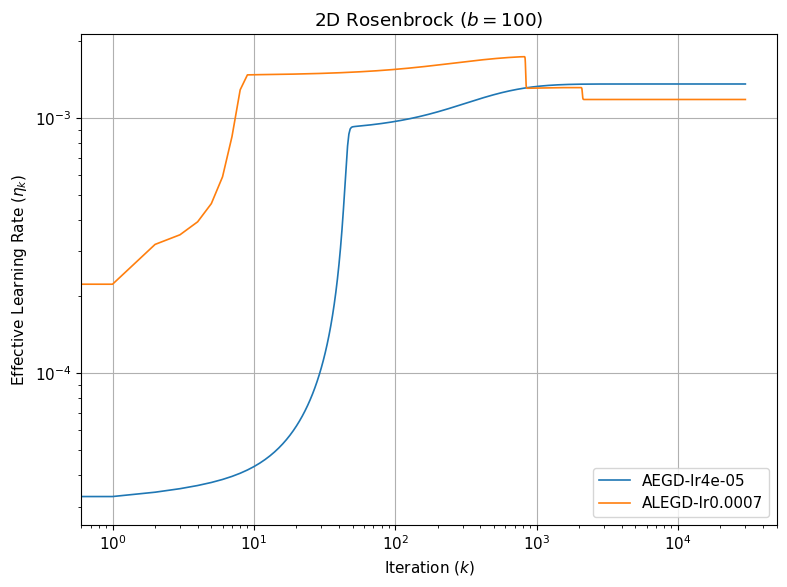}
    %\caption{}
    \end{minipage}%
    }%
    \centering
    \caption{Comparison of the behavior of the energy parameter ($r_k$) and the effective step size ($\eta_k$) for AEGD and ALEGD when applied to the 2D Rosenbrock objective with $b = 100$.}
    \label{rosen_two_stage}
\end{figure}

In Figure \ref{quad_two_stage_b},  the black dashed line denotes  the upper bound $\frac{2}{L}$ for the step size in standard GD when applied to the quadratic objective. Notably, after about $10$ iterations, the effective step sizes of both AEGD and ALEGD drop below this threshold. This  indicates that the updates become stable, showing consistent descent in the loss values. This trend aligns with the observations in Figure \ref{qual_alegd_loss}, where both AEGD and ALEGD begin showing decrease in loss around the 10th iteration. 

These observations highlight the two-stage adaptive nature of the step size in our algorithm: an initial phase with larger step sizes for rapid progress, followed by a more stable phase with a suitable step size that ensures convergence.

%%%%%%%%%%%%%%%%%%%%%%%%%%%
\subsubsection{ Effects of parameter $c$} \label{numerical_c}

Table \ref{c_qual} and \ref{c_rosen} present the impact  of the parameter $c$ on both the optimal base step size and the performance  of AEGD and ALEGD, measured by the number of iterations required to achieve a fixed accuracy of $10^{-7}$. In the tables, ``$lr$'' denotes the initial step size that gives the best performance (i.e., the fastest convergence), and ``$k$'' indicates  the corresponding number of iterations needed to achieve the target accuracy. 

\begin{table}[ht]
\centering
\caption{Optimal base step size ($lr$) and number of iterations ($k$) required to reach an accuracy of $10^{-7}$ for different values of $c$ on the 100D quadratic problem. Left: AEGD. Right: ALEGD.}
\begin{minipage}[t]{0.48\textwidth}
\centering
\begin{tabular}{ccc}
\toprule
$c$ & $lr$ & $k$ \\
\midrule
1    & 13  & 34 \\
10   & 27  & 23 \\
100  & 45  & 11 \\
1000 & 119 & 12 \\
\bottomrule
\end{tabular}

\vspace{0.3em}
{\small (a) AEGD}
\end{minipage}
\hfill
\begin{minipage}[t]{0.48\textwidth}
\centering
\begin{tabular}{ccc}
\toprule
$c$ & $\mathrm{lr}$ & $k$ \\
\midrule
1    & 17  & 53 \\
10   & 56  & 27 \\
100  & 94  & 19 \\
1000 & 131 & 20 \\
\bottomrule
\end{tabular}

\vspace{0.3em}
{\small (b) ALEGD}
\end{minipage}
\label{c_qual}
\end{table}

\begin{table}[ht]
\centering
\caption{Optimal base step sizes ($lr$) and iteration counts ($k$) required to reach an accuracy of $10^{-7}$ for different values of $c$ on the 2D Rosenbrock problem. Left: AEGD. Right: ALEGD.}
\begin{minipage}[t]{0.48\textwidth}
\centering
\begin{tabular}{ccc}
\toprule
$c$ & $lr$ & $k$ \\
\midrule
1    & $4.0\times 10^{-4}$ & 8035 \\
10   & $5.0\times 10^{-4}$ & 7281 \\
100  & $8.0\times 10^{-4}$ & 8028 \\
1000 & $2.9\times 10^{-3}$ & 9347 \\
\bottomrule
\end{tabular}

\vspace{0.3em}
{\small (a) AEGD}
\end{minipage}
\hfill
\begin{minipage}[t]{0.48\textwidth}
\centering
\begin{tabular}{ccc}
\toprule
$c$ & $lr$ & $k$ \\
\midrule
1    & $7.0\times 10^{-4}$ & 5465 \\
10   & $1.0\times 10^{-3}$ & 7765 \\
100  & $1.0\times 10^{-3}$ & 15000 \\
1000 & $1.1\times 10^{-3}$ & 18838 \\
\bottomrule
\end{tabular}

\vspace{0.3em}
{\small (b) ALEGD}
\end{minipage}
\label{c_rosen}
\end{table}

Our results show that, in general, as $c$ increases, the optimal step size tends to increase as well. However, a larger $c$ does not always lead to improved performance in terms of convergence speed. This highlights a trade-off between adaptivity and stability, which varies depending on the problem and algorithm. Therefore, we recommend tuning  $c$ carefully in practice. Selecting an appropriate value can help reduce the burden of fine-tuning the step size while maintaining robust convergence behavior.

\subsection{Experimental Evaluation} \label{experiments} 
We evaluate the proposed generalized AEGD framework, focusing on AEGD and ALEGD, across a range of optimization problems, highlighting their effectiveness, robustness, competitiveness, and practical limitations.  
%}

We begin with deterministic synthetic benchmarks, including quadratic and Rosenbrock functions, to provide baseline comparisons in Section \ref{section_synthetic}. We then evaluate logistic regression in Section \ref{section_LR}. Since noisy objectives naturally motivate stochastic optimization, we further test stochastic variants of the proposed methods on mini-batch logistic regression and image classification tasks in
Section \ref{section_image}, representing convex and nonconvex settings, respectively. Finally,  Section \ref{section_summary} summarizes the main findings.

We use a coordinate-wise implementation, updating the energy variable elementwise for each parameter. This variant has been found to improve empirical performance \cite{liu_aegd_2025}. Unless otherwise specified, we set $c = 1$ for both AEGD and ALEGD in all experiments.

We primarily compare our methods against GDM and Adam. For GDM, the momentum parameter is fixed at $\mu = 0.9$. For Adam, we use the standard hyperparameter settings $\beta_1 = 0.9$ and $\beta_2 = 0.999$. Across all methods, only the base learning rate is tuned,  while all other hyperparameters are kept fixed. 

For the synthetic optimization problems, the base step sizes are selected to achieve the fastest convergence in objective values. For logistic regression and image classification tasks,  learning rates are selected using a two-stage grid search consisting of a coarse search followed by a finer search around the best-performing candidate. For logistic regression, model selection is based on  training loss, whereas for image classification it is based on the highest test accuracy. 

\subsubsection{Baseline Comparison on Synthetic Benchmark Problems}\label{section_synthetic}
In this set of baseline experiments, we compare our methods not only with GDM and Adam, but also with gradient descent equipped with Armijo backtracking line search and the conjugate gradient (CG) method. 

For Armijo backtracking , we use an initial step size $\alpha_0 = 1$, a backtracking factor $\rho = 0.5$, and an Armijo constant $c_1 = 10^{-4}$, with a maximum of $50$ backtracking iterations per update. 

For the quadratic problem, we use the classical CG method with exact line search. For the Rosenbrock problem, we use the Fletcher--Reeves nonlinear CG method combined with Armijo backtracking line search. To improve robustness, a restart strategy is applied whenever the search direction ceases to be a descent direction.

\begin{figure}[h]
    \centering
    \subfigure{
    \begin{minipage}[t]{0.5\linewidth}
    \centering
    \includegraphics[width=2.5in]{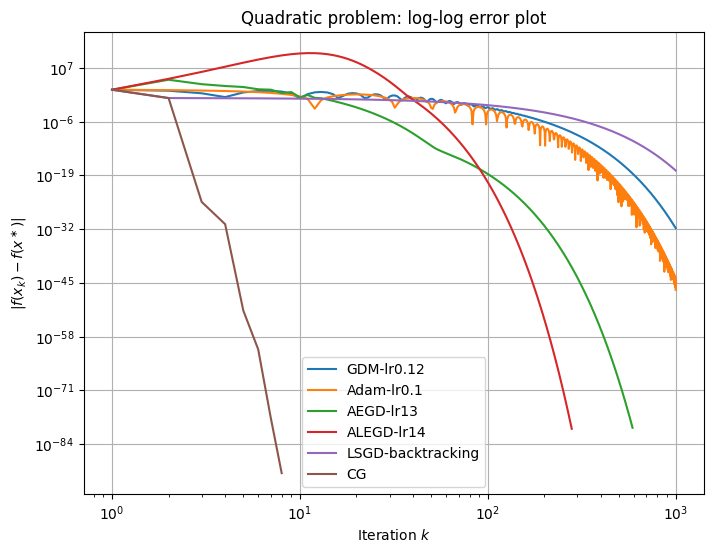}
    %\caption{}
    \end{minipage}%
    }%
    \subfigure{
    \begin{minipage}[t]{0.5\linewidth}
    \centering
    \includegraphics[width=2.5in]{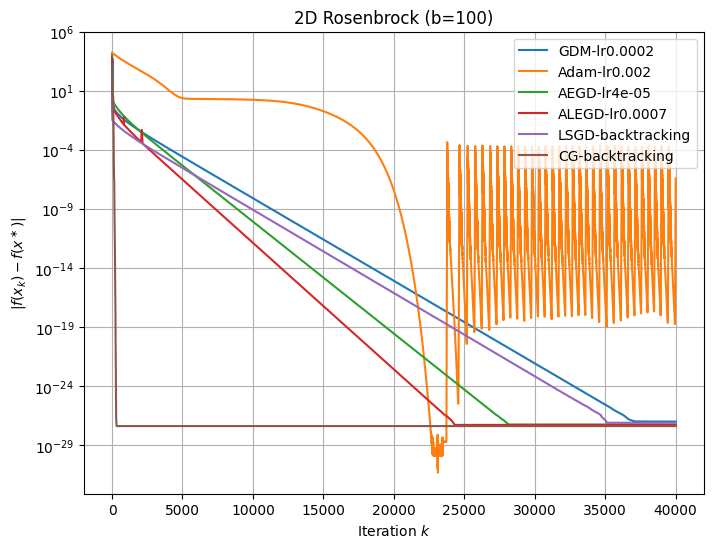}
    %\caption{}
    \label{quad_two_stage_db}
    \end{minipage}%
    }%
    \centering
    \caption{Loss comparison of GDM, GD with line search, Adam, AEGD and ALEGD on 100D quadratic objective and 2D-Rosenbrock objective with $b = 100$.}
    \label{synthetic}
\end{figure}

The resulting loss comparisons are presented in Figure \ref{synthetic}. As expected, the CG method achieves the fastest convergence on both benchmark problems, reflecting its ability to effectively exploit curvature information in structured optimization settings. 
Among the remaining  methods, AEGD and ALEGD exhibit noticeably faster reduction in objective error then GDM, Adam, and gradient descent with Armijo backtracking, supporting their optimization efficiency. On the Rosenbrock problem, Adam shows an extended  plateau followed by oscillatory behavior near the minimizer, indicating sensitivity to hyperparameter choices when using a fixed learning rate.

These results highlight the improved stability and the accelerated convergence of the proposed energy-based methods. They further suggest that AEGD and ALEGD provide robust and efficient alternatives to GDM, Adam, and GD with line search. At the same time, the proposed methods are not intended to compete with conjugate gradient methods on highly structured optimization problems, where curvature-exploiting techniques are known to be particularly effective.

\subsubsection{Logistic Regression} \label{section_LR} We next consider the standard $\ell_2$-regularized logistic regression problem
\begin{equation}
\min_{w \in \mathbb{R}^d} 
\; f(w) := \frac{1}{n} \sum_{i=1}^n \log\big(1 + \exp(-y_i x_i^\top w)\big)
+ \frac{\lambda}{2} \|w\|^2,
\end{equation}
where $(x_i, y_i) \in \mathbb{R}^d \times \{-1,1\}$ denote the training samples and $\lambda > 0$ is the regularization parameter.

We evaluate the proposed methods on several benchmark datasets from the LIBSVM\footnote{This is from the SVM library LIBSVM available at \url{https://www.csie.ntu.edu.tw/~cjlin/libsvmtools/datasets/}} collection. Specifically, we consider the deterministic setting for all algorithms on the datasets \texttt{a4a}, \texttt{a6a}, and \texttt{a9a}, while stochastic variants (described in Appendix \ref{s-gAEGD}) are evaluated on the large-scale dataset \texttt{w8a}. For each experiment, we report both the training loss and test accuracy, as shown in Figures~\ref{logistic_full} and~\ref{logistic_mini}. The corresponding fine-tuned learning rates are summarized in Table~\ref{tab:lr_logistic}.

All input features are standardized to have zero mean and unit variance, and a bias term is appended to each feature vector. All methods are initialized from the same random point $w_0 \sim 0.01 \mathcal{N}(0, I)$. 

\begin{figure}[h]
    \centering
    \includegraphics[width=1\linewidth]{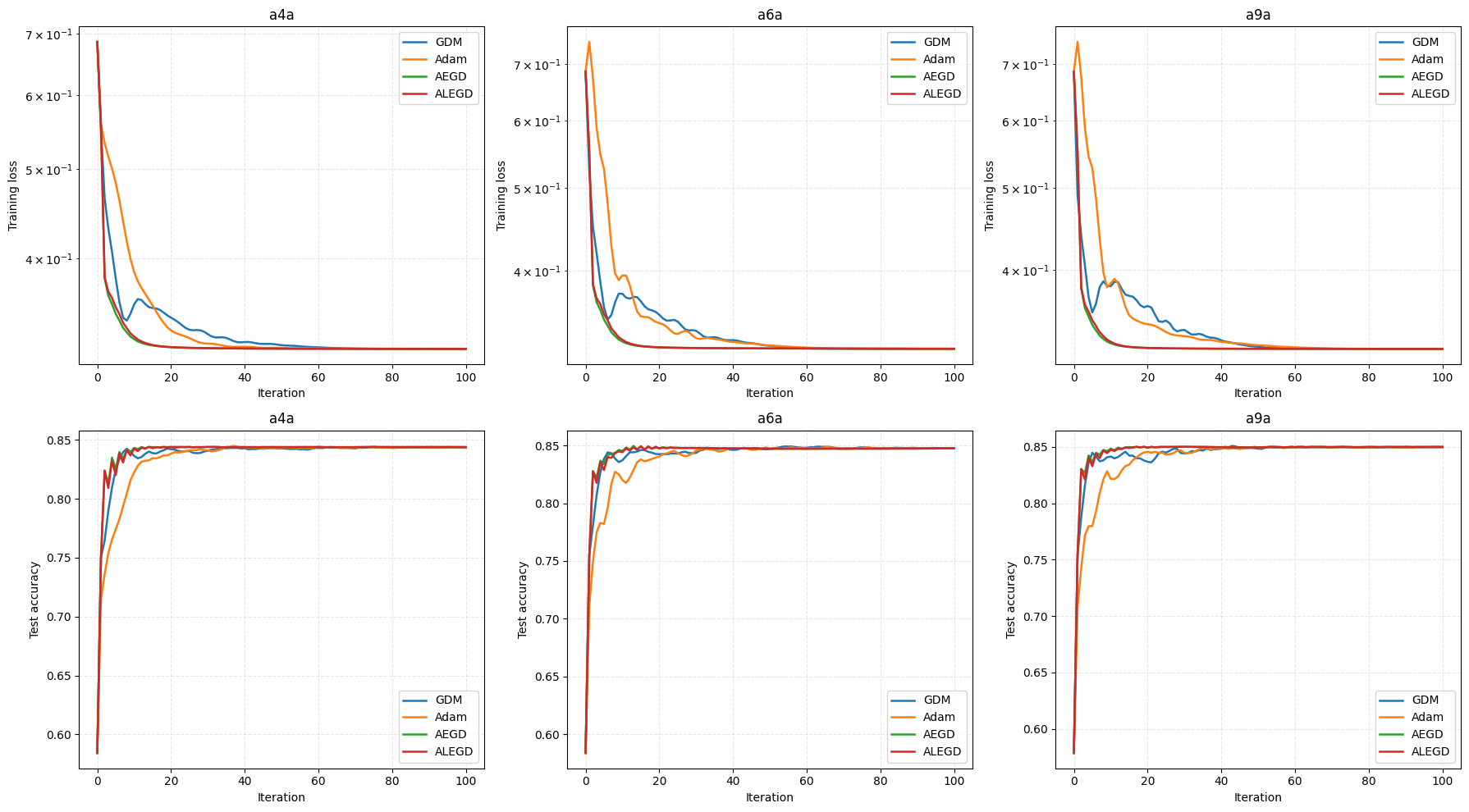}
    \caption{Training loss and test accuracy comparison for full batch logistic regression on datasets a4a, a6a, and a9a.}
    \label{logistic_full}
\end{figure}

The results in Figure \ref{logistic_full} show that, on the Adult datasets (a4a, a6a, a9a), all methods attain comparable final performance in terms of both training loss and test accuracy. However, AEGD and ALEGD converge substantially faster in terms of iteration count.

\begin{table}[ht]
\centering
\caption{Learning rates for logistic regression on LIBSVM datasets.}
\begin{tabular}{lcccc}
\hline
Dataset & SGDM & Adam & AEGD & ALEGD \\
\hline
a4a  & 0.5  & 0.1  & 3.0 & 3.0 \\
a6a  & 0.7  & 0.2  & 3.0 & 3.0 \\
a9a  & 1.0  & 0.2  & 3.0 & 3.0 \\
w8a  & 0.03 & 0.01 & 0.1 & 0.1 \\
\hline
\end{tabular}
\label{tab:lr_logistic}
\end{table}

\begin{figure}[h]
    \centering
    \subfigure{
    \begin{minipage}[t]{0.5\linewidth}
    \centering
    \includegraphics[width=3in]{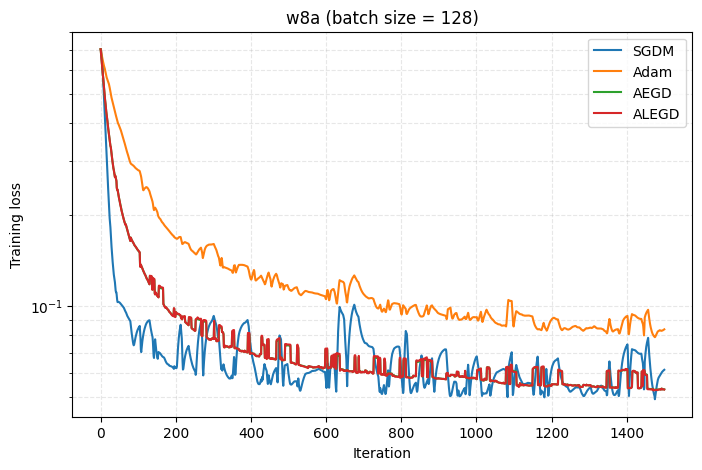}
    %\caption{}
    \end{minipage}%
    }%
    \subfigure{
    \begin{minipage}[t]{0.5\linewidth}
    \centering
    \includegraphics[width=3in]{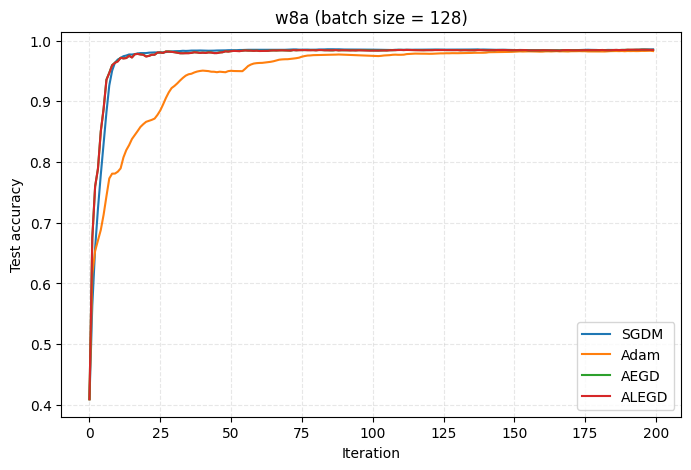}
    %\caption{}
    \label{quad_two_stage_db}
    \end{minipage}%
    }%
    \centering
    \caption{Training loss and test accuracy comparison for mini-batch logistic regression on dataset w8a. Test accuracy is shown only for the first 200 iterations to better visualize the early-stage convergence behavior.}
    \label{logistic_mini}
\end{figure}

On the w8a dataset with mini-batch training (Figure \ref{logistic_mini}), all methods achieve comparable final test accuracy. Nevertheless, SGDM, AEGD and ALEGD converge faster in terms of iteration count, whereas Adam exhibits slower progress during the early stages of training. In terms of training loss, SGDM shows rapid initial progress but also exhibits noticeable oscillatory behavior during training. Adam, while the most stable method over all, converges considerably more slowly on this task.  In contrast, AEGD and ALEGD achieve a more favorable balance, combining fast convergence with improved stability relative to SGDM. 

Overall, these results suggest that the proposed energy-based methods are well-suited for logistic regression problems. They offer an effective trade-off between convergence speed and optimization stability while maintaining competitive generalization
performance. 

\subsubsection{Image Classification}\label{section_image} We next compare  the proposed methods with SGDM and Adam on image classification tasks, including LeNet-5 on MNIST and ResNet-32 on CIFAR-10. 

For the MNIST experiments, all methods are trained for 50 epochs using a mini-batch size of 128 and a weight decay of $1 \times 10^{-4}$. 

For CIFAR-10, we use a fixed training budget of 200 epochs with weight decay $1 \times 10^{-4}$. The learning rate is decayed by a factor of $0.1$ after $150$ epochs. To evaluate robustness with respect to initialization and stochasticity, each method is run three times using random seeds $0$, $1$, and $2$. Reported results are presented as mean $\pm$ standard deviation across the three runs.

Figure \ref{MNIST} shows the convergence behavior of the different methods on LeNet-5 trained on MNIST, measured in terms of training loss and test accuracy over iterations. Table \ref{mnist} summarizes the corresponding final performance, including best and final test accuracy, as well as training and test losses. In the label, ``LR'' denotes the learning rate used for each method.

\begin{figure}[h]
    \centering
    \subfigure{
    \begin{minipage}[t]{0.5\linewidth}
    \centering
    \includegraphics[width=3in]{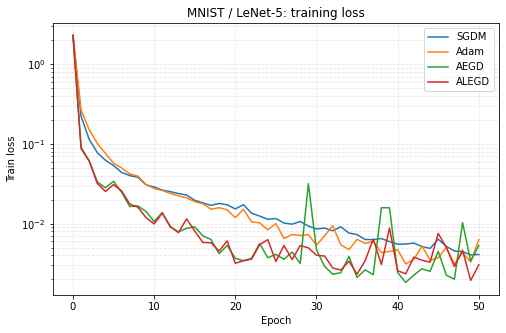}
    %\caption{}
    \end{minipage}%
    }%
    \subfigure{
    \begin{minipage}[t]{0.5\linewidth}
    \centering
    \includegraphics[width=3in]{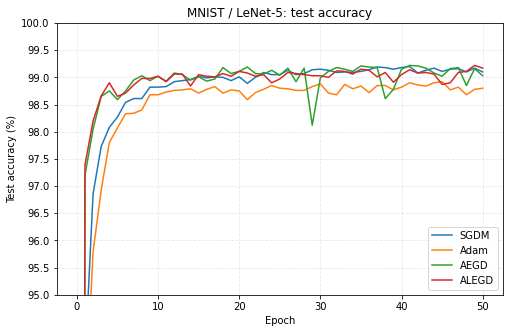}
    %\caption{}
    \label{quad_two_stage_db}
    \end{minipage}%
    }%
    \centering
    \caption{Training loss and test accuracy comparison for LeNet-5 on MNIST.}
    \label{MNIST}
\end{figure}

\begin{table}[ht] 
\centering
\caption{Final performance comparison on MNIST (LeNet-5).}
\begin{tabular}{lccccc}
\hline
Method & LR & Best Test Acc (\%) & Final Test Acc (\%) & Final Test Loss & Final Train Loss \\
\hline
SGDM  & 0.01   & 99.20 & 99.10 & \textbf{0.028469} & 0.004096 \\
Adam  & 0.0003 & 98.92 & 98.80 & 0.038549 & 0.006296 \\
AEGD  & 0.3    & \textbf{99.22} & 99.03 & 0.032532 & 0.005319 \\
ALEGD & \textbf{0.3}    & \textbf{99.22} & \textbf{99.17} & 0.028964 & \textbf{0.003043} \\
\hline
\end{tabular}
\label{mnist}
\end{table}
The results indicate that all methods achieve comparable final test accuracy on MNIST. However, AEGD and ALEGD exhibit rapid decrease in training loss and quicker improvement  in test accuracy during early stages of training. Although mild oscillations appear in later epochs, the energy-based methods remain competitive throughout training. Among them, ALEGD demonstrates improved stability and slightly better overall performance, achieving the highest test accuracy together with the lowest training loss. 

Figure \ref{CIFAR10} presents the convergence behavior of the methods on CIFAR-10 using ResNet-32, averaged over three independent runs. The corresponding quantitative results are summarized in
Table \ref{cifar10}. As shown,  all methods achieve similar final performance after the scheduled learning rate decay, reaching around 92–93\% test accuracy. However, clear differences are observed in both convergence behavior and final metrics. 

\begin{figure}[h]
    \centering
    \subfigure{
    \begin{minipage}[t]{0.5\linewidth}
    \centering
    \includegraphics[width=3in]{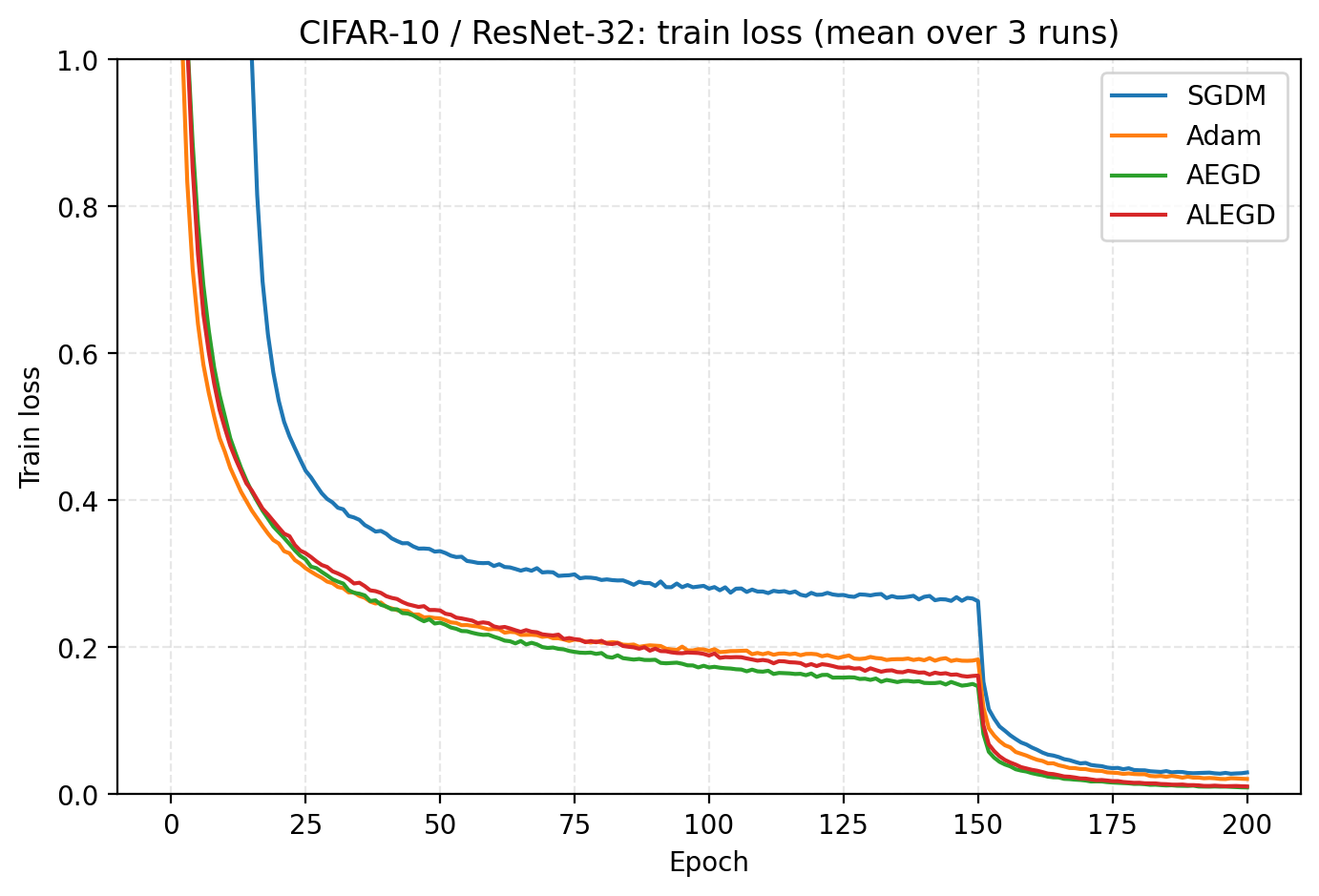}
    %\caption{}
    \end{minipage}%
    }%
    \subfigure{
    \begin{minipage}[t]{0.5\linewidth}
    \centering
    \includegraphics[width=3in]{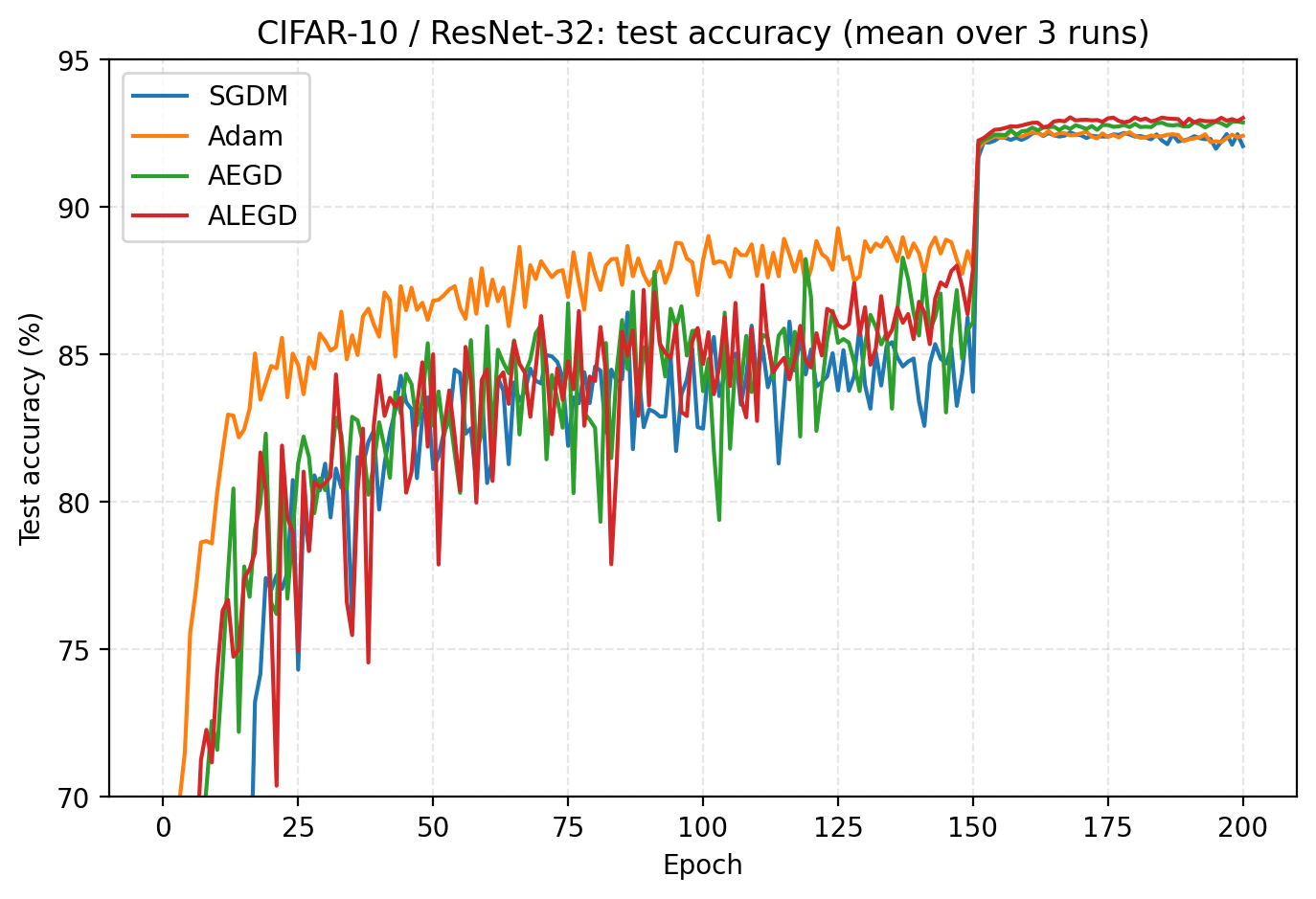}
    %\caption{}
    \label{quad_two_stage_db}
    \end{minipage}%
    }%
    \centering
    \caption{Training loss and test accuracy of ResNet-32 on CIFAR-10, averaged over three independent runs.}
    \label{CIFAR10}
\end{figure}

\begin{table}[ht]
\centering
\caption{Final performance comparison on CIFAR-10 (ResNet-32), reported as mean $\pm$ standard deviation over 3 runs.}
\begin{tabular}{lccccc}
\toprule
Method & LR & Best Test Acc (\%) & Final Test Acc (\%) & Final Test Loss & Final Train Loss \\
\midrule
SGDM  & 0.2   & $92.70 \pm 0.05$ & $92.07 \pm 0.30$ & $0.3511 \pm 0.0103$ & $0.0297 \pm 0.0020$ \\
Adam  & 0.002 & $92.73 \pm 0.10$ & $92.42 \pm 0.29$ & $0.3543 \pm 0.0021$ & $0.0209 \pm 0.0003$ \\
AEGD  & 0.2   & $93.01 \pm 0.08$ & $92.87 \pm 0.18$ & $0.3020 \pm 0.0224$ & $\mathbf{0.0094 \pm 0.0005}$ \\
ALEGD & 0.7   & $\mathbf{93.15 \pm 0.11}$ & $\mathbf{93.03 \pm 0.21}$ & $\mathbf{0.2963 \pm 0.0177}$ & $0.0108 \pm 0.0006$ \\
\bottomrule
\end{tabular}
\label{cifar10}
\end{table}

In terms of training loss, SGDM exhibits a relatively slow decay, although it becomes comparable to others after learning rate decay. AEGD and ALEGD show a similar convergence behavior to Adam, while achieving lower final training loss in this task. 

For test accuracy, before learning rate decay, Adam demonstrates faster initial progress, whereas SGDM, AEGD, and ALEGD converge more slowly and exhibit some oscillations. After learning rate decay, all methods show similar convergence trends, with AEGD and ALEGD ultimately achieving better best and final generalization performance than SGDM and Adam in this task.

Moreover, the small standard deviations indicate that the proposed methods are also stable across initializations. Notably, ALEGD operates effectively with a significantly larger learning rate, highlighting its robustness and adaptivity. 

\subsubsection{Summary}\label{section_summary}
%\red{
Overall, the experiments indicate three regimes: (i) on deterministic or moderately stochastic problems, AEGD/ALEGD often converge faster than SGDM and Adam; (ii) in deep learning tasks, they achieve comparable or slightly better generalization with improved stability; (iii) on highly structured problems, curvature-based methods such as CG remain superior. This clarifies the settings where the proposed approach is most effective in practice.

\section{Discussion}
\label{conclusion}
We introduced a broad class of energy-based adaptive gradient methods, extending  the AEGD methodology within a unified framework \eqref{updating rule}, where the energy functions satisfying Assumption \ref{assumptions on F}. Our theoretical results show that the proposed framework exhibits unconditionally energy stability, meaning that energy stability is ensured  regardless of the choice of the base step size. We identify a two-stage behavior where the effective step size first adjusts adaptively,  then stabilizes within a range that guarantees the decay of the objective values. In addition, we establish sufficient conditions on the base step size to ensure convergence and  analyze the method's robustness with respect to its choice. Furthermore, we prove that when the parameter $c$ is chosen sufficiently large, convergence is guaranteed even when step size $\eta$ is not small. For objective functions that also satisfy the KL condition, we derive explicit convergence rates, 
strengthening the theoretical 
guarantees of our framework.

Our experimental results consistently support both the theoretical properties and practical advantages of the proposed energy-based adaptive framework. Across a range of benchmark problems, particularly in non-convex settings, the experiments confirm the predicated two-stage behavior of both the energy parameter $r_k$ and the effective step size $\eta_k$. 
In addition, we observe that increasing the parameter $c$ can relax the selection of the base step size $\eta$, although this does not necessarily lead to improved performance. As such,  we recommend tuning  $c$ appropriately in practice to reduce the effort required for fine-tuning $\eta$.

%\red{
In terms of practical performance, AEGD and ALEGD generally achieve competitive results across a range of optimization tasks, and in many cases converge faster. However, oscillations may occur in some settings, particularly before learning rate decay. The methods are also not designed to compete with curvature-based approaches such as conjugate gradient on structured problems. Our
experiments further suggest that the logarithmic energy function used in ALEGD can perform more consistently than the square-root energy in AEGD across several benchmarks. Overall, these results indicate that the generalized energy-based approach can serve as a competitive and practical alternative
to existing first-order optimizers in a variety of settings.
%}

The main limitation of our study lies in the absence of a clear guideline for selecting the energy function. Although we proposed a family of algorithms based on various choices of energy functions and observed in preliminary experiments that ALEGD often outperforms AEGD,  a definitive principle for selecting the most suitable energy function for a given objective remains to be established. 

Moreover, our methods belong to the class of local optimization algorithms, and without any intervention or additional conditions, they may get stuck at saddle points. Regarding the saddle points, the Stable Manifold Theorem from dynamical systems theory (as applied in works such as \cite{lee_first-order_2019} and \cite{panageas_gradient_2017}) implies that if the function $f$ satisfies the strict saddle property (i.e., every saddle point has at least one direction with a  strictly negative eigenvalue), then the set of initial conditions that converge to a saddle point has measure zero. In practice, convergence to saddle points can often be avoided by introducing a small amount of noise into the gradient updates \cite{jin_how_2017}. 

\appendix
\section{Stochastic generalized energy-based adaptive gradient method (S-gAEGD)}\label{s-gAEGD}

\begin{algorithm}[h]
\caption{Stochastic Generalized Energy-Based Adaptive Gradient Method (S-gAEGD)}
\label{alg:sgaegd}
\begin{algorithmic}[1]
\Require Initial point $x_0 \in \mathbb{R}^d$, step size $\eta>0$, constant $c>0$, energy function $\hat F:\mathbb{R}_{+}\to\mathbb{R}_{+}$, total number of iterations $T$
\For{$k=0,1,\dots,T-1$}
    \State Sample a mini-batch $B_k$
    \State Compute stochastic gradient $g_k \leftarrow \nabla f_{B_k}(x_k)$ and mini-batch loss $f_k^b \leftarrow f_{B_k}(x_k)$
    \State Set 
    \[
    F_k \leftarrow \hat F(f_k^b + c), \quad
    F_k' \leftarrow \hat F'(f_k^b + c)
    \]
    \If{$k=0$}
        \State Set $r_k \leftarrow F_k$
    \EndIf
    \State Update the energy variable:
    \[
    r_{k+1}
    \leftarrow
    \frac{r_k}{1+\eta \frac{F_k'}{F_k}\|g_k\|^2}
    \]
    \State Update the iterate:
    \[
    x_{k+1}
    \leftarrow
    x_k-\eta \frac{r_{k+1}}{F_k} g_k
    \]
\EndFor
\State \Return $x_T$
\end{algorithmic}
\end{algorithm}

\bibliographystyle{unsrt}
\bibliography{refs} 

\end{document}